\documentclass{article}   	
\usepackage{geometry}               
\usepackage{graphicx}
\usepackage[usenames,dvipsnames]{xcolor}
\usepackage{amssymb,amsmath,amsthm,amsfonts}

\usepackage[shortlabels]{enumitem}

\usepackage{hyperref} 
\usepackage[capitalise, nameinlink, noabbrev]{cleveref} 
\hypersetup{colorlinks=true, 
	linkcolor=blue,
	citecolor=red,
}
\allowdisplaybreaks

\usepackage[english]{babel}
\usepackage[utf8]{inputenc}

\usepackage{mathtools}
\usepackage{esint} 
\usepackage{dsfont}
\usepackage{xfrac}

\usepackage[sc]{mathpazo}

\usepackage{tikz}
\newtheorem{proposition}{Proposition}[section]
\newtheorem{theorem}[proposition]{Theorem}
\newtheorem{corollary}[proposition]{Corollary}
\newtheorem{lemma}[proposition]{Lemma}

\newtheorem{assumptions}[proposition]{Assumptions}
\theoremstyle{definition}
\newtheorem{definition}[proposition]{Definition}
\theoremstyle{remark}
\newtheorem{remark}[proposition]{Remark}
\numberwithin{equation}{section}

\newcommand{\eps}{\varepsilon}

\newcommand{\N}{{\mathbb{N}}}

\newcommand{\R}{{\mathbb{R}}}

\DeclareMathOperator{\dist}{dist}

\DeclareMathOperator{\diverg}{div}

\title{On the boundary branching set of the one-phase problem}
\author{Lorenzo Ferreri, Luca Spolaor, Bozhidar Velichkov}

\begin{document}
\maketitle

\begin{abstract}
    We consider minimizers of the one-phase Bernoulli free boundary problem in domains with analytic fixed boundary. In any dimension $d$, we prove that the branching set at the boundary has Hausdorff dimension at most $d-2$. As a consequence, we also obtain an analogous estimate on the branching set for solutions to the two-phase problem under an analytic separation condition. Moreover, as a byproduct of our analysis we obtain strong boundary unique continuation results for quasilinear operators and thin-obstacle variational inequalities. The approach we use is based on the (almost-)monotonicity of a boundary Almgren-type frequency function, obtained via regularity estimates and a Calder\'on-Zygmund decomposition in the spirit of Almgren-De Lellis-Spadaro. 
\end{abstract} 
\noindent
{\footnotesize \textbf{AMS-Subject Classification}}. 
{\footnotesize 35R35, 	
35J57.   
}

\tableofcontents

\section{Introduction}
Let $\mathcal A$ be a bounded open set in $\R^d$, $d\ge 2$, and $u_0:\partial \mathcal{A}\to\R$ be a given nonnegative function. The one-phase Bernoulli problem consists in minimizing the functional 
\begin{equation}\label{eqn:onePhaseEnergyU-A}
J_1(u,\mathcal A) \coloneqq \int_{\mathcal A} \vert \nabla u \vert^2\,dx + \vert 
\Omega_u \cap \mathcal A \vert \quad 
\end{equation}
among all nonnegative functions $u:\overline {\mathcal A}\to\R$ that agree with $u_0$ at $\partial \mathcal A$, where for any function $u\in H^1(\mathcal A)$ we denote by $\Omega_u$ its positivity set 
$$\Omega_u:=\{u>0\}\cap \mathcal A.$$ 
It is well-known, due to the works \cite{AltCaffarelli:OnePhaseFreeBd, DeSilva:FreeBdRegularityOnePhase, CaffarelliJerisonKenig04:NoSingularCones3D, JerisonSavin15:NoSingularCones4D, DeSilvaJerison09:SingularConesIn7D, Weiss99:PartialRegularityFreeBd, KinderlehrerNirenberg1977:AnalyticFreeBd}, that the optimal regularity for minimizers $u$ of \eqref{eqn:onePhaseEnergyU-A} is Lipschitz and that the free boundary $\partial\Omega_u\cap \mathcal A$ inside $\mathcal A$ is analytic up to a closed singular set of Hausdorff dimension at most $d-5$. The behavior of the free boundary $\partial \Omega_u\cap \mathcal A$ and the solution $u$ up to the boundary of $\partial \mathcal A$ was first studied in \cite{ChangLaraSavin:BoundaryRegularityOnePhase} around points $z_0\in\partial A$ at which the boundary datum $u_0$ vanishes identically:
$$u_0\equiv 0\quad\text{in}\quad B_r(z_0)\cap \mathcal A\,,$$
for some $r>0$. Without loss of generality, we can take $B_r(z_0)=B_1$. It is immediate to check that if $u_0\equiv 0$ in $\partial \mathcal A\cap B_1$, then the minimizer $u$ (trivially extended in $B_1\setminus \mathcal A$) is a solution to the variational problem
\begin{equation}\label{e:intro-onePhase-variational-problem}
\min\Big\{J_1(v,B_1)\ :\ v \ge0,\ v=u\ \text{on}\ \partial B_1,\ \{v>0\}\subset \mathcal A\Big\}.
\end{equation}
In \cite{ChangLaraSavin:BoundaryRegularityOnePhase} it was shown that if $u$ is a solution to the above variational problem and $\partial \mathcal A$ is $C^{1,\sfrac12}$ smooth in a neighborhood of $z_0$, then the free boundary $\partial\Omega_u \cap \mathcal{A}$ is $C^{1,\sfrac12}$ regular in a neighborhood of any point on $\partial\Omega_u\cap\partial\mathcal A$; this regularity is also optimal in the sense that even for analytic $\partial\mathcal A$, there are solutions $u$ whose boundary $\partial\Omega_u$ is no more than $C^{1,\sfrac12}$ (see \cite{DePhilippisSpolaorVeluchkov2021:QuasiConformal2D, DaivdEngelsteinToroSmitVega2023:BranchPointsAlmostTwoPhase}). Moreover, any minimizer $u$ is a (classical) solution to the problem 
\begin{equation}\label{eqn:onePhase-equation}
\begin{cases}
\begin{array}{rcl}
\Delta u=0&\quad\text{in}\quad &\Omega_u\cap\mathcal A\\
u=0&\quad\text{on}\quad &\partial\Omega_u\\
|\nabla u|=1&\quad\text{on}\quad &\partial\Omega_u\cap \mathcal A\\
|\nabla u|\ge 1&\quad\text{on}\quad &\partial\Omega_u\cap \partial\mathcal A\,,
\end{array}
\end{cases}
\end{equation}
in a neighborhood of the contact set $B_1\cap\partial\Omega_u\cap\partial\mathcal A$. \medskip

The points where the two boundaries $\partial\Omega_u$ meet 
 $\partial\mathcal A$ are the so called {\it points of branching}:
\begin{equation}\label{e:intro-set-of-branching-points}
\mathcal B_1(u):=\Big\{x\in B_1\cap \partial\Omega_u\cap\partial\mathcal A\ :\ \left|B_r(x)\cap  \left(\mathcal A\setminus \Omega_u\right)\right|\neq 0\ \text{ for every }\ r>0\Big\}.
\end{equation}
The $C^{1,\sfrac12}$ regularity of $\partial\mathcal A$, $\partial\Omega_u$ and $u$ provides that at the points of branching the two boundaries $\partial\Omega_u$ and $\partial\mathcal A$ are tangent and the minimizer $u$ satisfies
 $$u=0\quad\text{and}\quad |\nabla u|=1\quad\text{at}\quad \mathcal B_1(u),$$
so at first order these points are indistinguishable from the interior points $\partial\Omega_u\cap\mathcal A$. Thus, the ($C^{1,\sfrac12}$-)regularity of $\partial\Omega_u$ and $\partial\mathcal A$ by itself does not provide any a priori information on the contact set $\partial\Omega_u\cap\partial\mathcal A$ and the way $\partial\Omega_u$ is approaching $\partial\mathcal A$.
\medskip

This branching behavior is a common issue in the free boundary regularity theory and in geometric analysis; it appears naturally at points at which the first order blow-up analysis of the solutions does not provide any geometric information about the set of points itself. 

Determining the fine structure of these branching points is not an easy task even in dimension $d=2$ and only few results are available in this direction:
we refer to \cite{Lewy1972:CoincidenceSetVariationalInequalitites} for the case of the thin-obstacle problem, \cite{Sakai1991:RegularityBoundarySchwarzFunction, Sakai1993:RegularityFreeBoundaries2D} for the case of the obstacle problem, and to \cite{DeLellisSpadaroSpolaor2018:2DMinialCurrentsILipschitz, DeLellisSpadaroSpolaor2017:2DMinialCurrentsIICenterManifold, DeLellisSpadaroSpolaor2017:2DMinialCurrentsIIIBlowUp, Chang1988:AreaMinimizingCurrents2D,DLboundary} for similar results in the framework of area-minimizing currents.
Concerning the Bernoulli problem, in \cite{DePhilippisSpolaorVeluchkov2021:QuasiConformal2D, FerreriSpolaorVeluchkov2024:FineStructureThinTwoMembrane2D} it was recently proved that, in dimension $d=2$ and for analytic $\mathcal A$, the set of branching points is locally finite; the results in \cite{DePhilippisSpolaorVeluchkov2021:QuasiConformal2D, FerreriSpolaorVeluchkov2024:FineStructureThinTwoMembrane2D} being obtained through a quasi-conformal map argument.   

The analysis of the branching sets in higher dimension is strongly related to a unique continuation type problem, and therefore to Almgren's frequency function. In the framework of area-minimizing currents, the upper bound on the dimension of the branching set in the interior is known thanks to the monumental work of Almgren \cite{Almgren2000:AlmgrensBigRegularityPaper} and the subsequent works of De Lellis and Spadaro \cite{DeLellisSpadaro2016:AreaMinimizingCurrents1LpEstimates, DeLellisSpadaro2016:AreaMinimizingCurrents2CenterManifold, DeLellisSpadaro2016:AreaMinimizingCurrents3BlowUp}, while the branching behavior at the boundary, for analytic boundaries, is currently an open problem (see \cite{White}). For the linear thin-obstacle problem, where the Almgren's frequency function is well-known to hold (see \cite{AthanasopoulosCaffarelli:SignoriniPb}), a $(d-2)$-rectifiability result was recently obtained by Focardi and Spadaro in \cite{FocardiSpadaro2018:MeasureFreeBdThinObstacle, FocardiSpadaro2018:MeasureFreeBdThinObstacleCorrection}. For the obstacle problem in dimension $d>2$ it is known, thanks to the works of Caffarelli \cite{Caffarelli1998:ObstacleProblemRevisited} and Monneau \cite{Monneau2003:ObstacleProblem2D} (see also \cite{ColomboSpolaorVelichkov2018:EpiperimatricObstacleProblem,FigalliSerra2019:FineStructureObstacleProblem}), that the set of singular points (which exhibits the same type of branching behavior) is contained in a $C^1$ manifold of dimension $(d-1)$. While for generic free boundaries finer results were obtained by Figalli, Ros-Oton and Serra \cite{FigalliRos-OtonSerra2020:GenericRegularityObstacle}, the optimal dimension of the branching set is still open. \medskip

The aim of the present paper is to give a more precise description of the set of branching points \eqref{e:intro-set-of-branching-points} in any dimension $d\ge 2$ for the one-phase Bernoulli problem at the boundary. We extend the results from \cite{DePhilippisSpolaorVeluchkov2021:QuasiConformal2D, FerreriSpolaorVeluchkov2024:FineStructureThinTwoMembrane2D} to any dimension by showing that, when $\partial\mathcal A$ is a $(d-1)$-dimensional analytic manifold, the set of branching points $\mathcal B_1(u)$ has Hausdorff dimension at most $(d-2)$. We notice that the analyticity of $\partial\mathcal A$ is fundamental for this result. Indeed, one can easily produce examples of wildly behaving contact sets by  taking a half-plane solution $u=(x_d)^+$ and then constructing a set $\mathcal A$ with $C^\infty$ boundary $\partial\mathcal A$ touching   $\partial\Omega_u=\{x_d=0\}$ on an arbitrary closed set. \medskip

In what follows, we assume that the origin is a branching point, $0\in\mathcal B_1(u)$. We will denote by $x'$ the points in $\R^{d-1}$, so that $B_1':=B_1\cap (\R^{d-1}\times\{0\})$, and we assume that the boundary of $\mathcal A$ is the graph of an analytic function 
$\phi:B_1'\to\R\,$, precisely:
\[
\mathcal{A} \coloneqq \{(x',x_d)\in B_1\ :\ x_d>\phi(x')\}\,.
\]
Thanks to the $C^{1,\sfrac12}$ regularity of $\partial\Omega_u$ we may assume that there is a $C^{1,\sfrac12}$ function
	$$f:B_1'\to\R\ ,\quad f\ge\phi\ \text{ on }\  B_1',$$
	such that, up to a rotation and translation of the coordinate system, we have 
	\begin{equation}\label{eq:fbgraph}
	\begin{cases}
	\begin{array}{rcl}
	u(x)>0&\quad\text{for}\quad&x\in (x',x_d)\in B_1 \quad\text{such that}\quad x_d>f(x');\\
	u(x)=0&\quad\text{for}\quad&x\in (x',x_d)\in B_1 \quad\text{such that}\quad  x_d\le f(x').
	\end{array}
	\end{cases}
	\end{equation}	
In terms of the functions $f$ and $\phi$ the contact set of the two boundaries $\partial\Omega_u$ and $\partial\mathcal A$ reads as 
\begin{equation}\label{e:def-C-one-phase}
\mathcal C_1(u):=B_1\cap\partial \Omega_u\cap\partial\mathcal A=\{(x',x_d)\in B_1\ :\ x_d=\phi(x')=f(x')\}\,.
\end{equation}
We also introduce the set of points
\begin{equation}\label{e:def-S-one-phase}
\mathcal S_{1}(u):=\big\{x\in \mathcal C_1(u)\,:\,|\nabla u|(x)=1\big\},
\end{equation}
which is a closed subset of $\{x_d=\phi(x')\}$ and contains all points of branching, that is:
$$\mathcal B_1(u)\subset \mathcal S_1(u).$$
The following is the main result of the paper.
    \begin{theorem}[Dimension of the boundary branching set]\label{thm:onePhaseAnalyticObstacle} Let $B_1\subset \R^d$ and let $u\in H^1(B_1)$ be a solution of \eqref{e:intro-onePhase-variational-problem}. Suppose, moreover, that the function $\phi$ describing $\partial\mathcal A$ is analytic. Then, either
    \[
    \partial \Omega_u \cap B_1 \equiv \mathrm{graph}(\phi) \cap B_1
    \]
    or
    \[
    \dim_{\mathcal H}(\mathcal S_1(u))\leq d-2.
    \]
    Moreover, in the second case, if $d=2$ then $\mathcal S_1(u)$ is locally finite.
    \end{theorem}
Differently to \cite{DePhilippisSpolaorVeluchkov2021:QuasiConformal2D} and \cite{FerreriSpolaorVeluchkov2024:FineStructureThinTwoMembrane2D} where the arguments rely on quasi-conformal maps techniques, the present paper uses an approach more similar to \cite{Almgren2000:AlmgrensBigRegularityPaper, DeLellisSpadaro2016:AreaMinimizingCurrents1LpEstimates, DeLellisSpadaro2016:AreaMinimizingCurrents2CenterManifold, DeLellisSpadaro2016:AreaMinimizingCurrents3BlowUp}
and \cite{Lin1991:NodalSetsEllipticParabolicPDEs, AdolfssonVilhelmEscurianza1997:C1aBoundaryUniquecontinuation, KukavicaNystrom1998:UniqueContinuationDiniBoundary}.
\begin{remark}
We notice that the estimate of Theorem \ref{thm:onePhaseAnalyticObstacle} is optimal in all dimensions. Indeed, in dimension two, examples of solutions with isolated branching points were constructed in \cite{DePhilippisSpolaorVeluchkov2021:QuasiConformal2D}. These 2D solutions can be then used to build sharp examples in any dimension $d$, by extending them to functions invariant with respect to the remaining $d-2$ variables. The extensions obtained this way are still solutions to \eqref{eqn:onePhase-equation} and minimizers to \eqref{e:intro-onePhase-variational-problem} (see \cite[Appendix A]{DePhilippisSpolaorVeluchkov2021:QuasiConformal2D}). 
\end{remark}
As a consequence of our analysis, we also obtain three results of independent interest:
\begin{enumerate}
    \item a boundary unique continuation result for quasilinear elliptic operators (\cref{thm:boundaryuniquecontinuation} in \cref{subsection:unique-continuation});

    \item an estimate on the dimension of the free boundary in quasilinear thin-obstacle problems (\cref{thm:boundaryuniquecontinuation2} in \cref{subsection:unique-continuation});

    \item an estimate on the dimension of the set of branching points in the two-phase problem, under the assumption that there exists an analytic manifold lying between the two free boundaries (\cref{theorem:twoPhaseSymmetricBranchingEst} in \cref{section:BranchinSetTwoPhaseAnalyticSeparation}). 
\end{enumerate}

\subsection{Main ideas and sketch of the proof}
In this subsection we describe the main ideas of the proof, discussing what are the major difficulties and new ideas to overcome them, compared to previous work. The proof of \cref{thm:onePhaseAnalyticObstacle} is essentially divided in the following $3$ steps.

\medskip

\noindent\textbf{Step 1. An $m$-hodograph transform.} In the first step we use the analyticity of $\phi$ to construct in a neighborhood of a singular branching point a solution to the one-phase Bernoulli problem $m$ with $\Omega_m=\mathcal A$. We then use $m$ to construct a new change of coordinates $\Phi$ (see \cref{prop:changeofcoord}) which turns the problem into a boundary unique continuation type problem for a quasilinear operator in $B_1^+$ with Robin boundary conditions on $B_1'$ (see \cref{thm:thinobs}).

One way to compare the free boundaries $\partial\Omega_u$ and $\partial\Omega_m$ is to apply an hodograph transform to $u$ and $m$ separately, the transformed functions $h_u$ and $h_m$ being simply the parametrizations of the graphs of $u$ and $m$ in the direction $e_d$. Then each of the linearized functions $h_u-x_d$ and $h_m-x_d$ is a local minimizer of the functional 
$$\varphi\mapsto \int_{B_1^+}\frac{|\nabla \varphi|^2}{1+\partial_d\varphi}.$$
It is immediate to check that the resulting PDE for the difference $h_u-h_m$ contains terms of any order (including linear terms in the gradient) and cannot be written as an autonomous nonlinear equation for $h_u-h_m$. To overcome this difficulty, we introduce a new change of coordinates $\Phi$ which is the key tool in this section. Precisely, we construct $\Phi$ through a gradient flow, which maps the level sets of $m$ into those of the function $x_d$.

We combine the change of coordinates $\Phi$ with the hodograph transforms of $u\circ\Phi$ and $m\circ\Phi$, thus obtaining an hodograph-type transform relative to $m$, which we call $m$-hodograph transform. Since $m\circ\Phi=x_d$ and since $m$ is a solution of the one-phase Bernoulli problem, we are able to obtain that the difference function $w$ (coinciding with the linearization of the transform of $u\circ\Phi$) minimizes a functional of the form
\begin{equation}\label{e:intro-functional-w}
w\mapsto \int_{B^+_r} \Big(M(x) \nabla w\cdot \nabla w + \partial_d Q(x)\,w\,\partial_dw+e(x,w,\nabla w)\Big),
\end{equation}
where the error term $e$ is analytic in $w$ and $\nabla w$ and contains only terms of order $\ge3$ (for the precise expression of the above functional see \cref{thm:thinobs}). In particular, $w$ is a solution to a quasi-linear thin-obstacle problem with Robin boundary conditions, whose energy does not have linear terms in the gradient $\nabla w$ or the function $w$. 

The geometric meaning of the above $m$-hodograph transform is hidden in the fact that it allows to compare the level sets of the original solution $u$ and the one-phase extension $m$, instead of comparing the functions (or their classical hodograph transforms) themselves. 

In the framework of the works of Almgren \cite{Almgren2000:AlmgrensBigRegularityPaper} and De Lellis-Spadaro \cite{DeLellisSpadaro2016:AreaMinimizingCurrents1LpEstimates, DeLellisSpadaro2016:AreaMinimizingCurrents2CenterManifold, DeLellisSpadaro2016:AreaMinimizingCurrents3BlowUp} on area-minimizing currents, the $m$-hodograph transform is reminiscent of a reparametrization over the center manifold. While the construction of the center manifold in \cite{Almgren2000:AlmgrensBigRegularityPaper, DeLellisSpadaro2016:AreaMinimizingCurrents2CenterManifold} is more complicated than in our setting, the normal reparametrization over the center manifold is natural and doesn't require further change of coordinates. In our case however, how to compare the functions $u$ and $m$ is not clear a priori, since $u$ satisfies both an interior and a boundary equation, which are not invariant under geometric reparametrizations.

\medskip

\noindent \textbf{Step 2. (Almost-)monotonicity of the frequency function.} We now wish to show that minimizers $w$ of the functional from \eqref{e:intro-functional-w} have almost-monotone frequency function (see \cref{thm:freqmon} below). We remark that the functional under consideration does not fall in any of the existing interior or boundary unique continuation type theorems (for instance \cite{GarofaloLin1987:UniqueContinuationFrequency, AdolfssonVilhelmEscurianza1997:C1aBoundaryUniquecontinuation, Lin1991:NodalSetsEllipticParabolicPDEs, KukavicaNystrom1998:UniqueContinuationDiniBoundary}). In fact, on one side, the function $w$ is a solution of an elliptic problem of the form
$${\rm div}(A(x,w,\nabla w)\nabla w)=f(x,w,\nabla w),$$
where the matrix field $A$ is analytic in $w$ and $\nabla w$ and contains terms of order 1 in $\nabla w$. On the other hand, the optimal regularity of $u$ (and consequently of $w$) is known to be exactly $C^{1,\sfrac12}$ (see \cite{ChangLaraSavin:BoundaryRegularityOnePhase,DePhilippisSpolaorVeluchkov2021:QuasiConformal2D}). This means that the matrix field 
$$x\mapsto A\left(x,w(x),\nabla w(x) \right)$$
is no more that $\sfrac12$-H\"older continuous, so we cannot apply the theory of Garofalo and Lin \cite{GarofaloLin1987:UniqueContinuationFrequency}, which requires Lipschitz regular matrix fields.\medskip

\noindent To overcome these difficulties, we proceed with the following main steps:
\begin{enumerate}
    \item[(i)] a $C^{1,\alpha}$ regularity result with estimates for minimizers $w$ of \eqref{e:intro-functional-w} (\cref{prop:nonlinearThinObstacleC1aRegularity});
    \item[(ii)] unique continuation type lemmas for minimizers of \eqref{e:intro-functional-w} under doubling assumptions (\cref{lemma:excessOnePhaseThreeAnnuliLemma} and \cref{lemma:heightOnePhaseThreeAnnuliLemma});
    \item[(iii)] a Taylor expansion of inner and outer variations (\cref{lem:freqidentities});
    \item[(iv)] estimates of the errors in the Taylor expansions from (iii), obtained combining the $C^{1,\alpha}$ regularity from (i) and the unique continuation lemmas from (ii), with a Calder\'on-Zygmund type decomposition similar to the one of Almgren-De Lellis-Spadaro \cite{Almgren2000:AlmgrensBigRegularityPaper, DeLellisSpadaro2016:AreaMinimizingCurrents1LpEstimates, DeLellisSpadaro2016:AreaMinimizingCurrents2CenterManifold, DeLellisSpadaro2016:AreaMinimizingCurrents3BlowUp} constructed in such a way that the doubling conditions required in point (ii) are satisfied.
\end{enumerate}
Notice that, while the underlying strategy is reminiscent of \cite{Almgren2000:AlmgrensBigRegularityPaper, DeLellisSpadaro2016:AreaMinimizingCurrents1LpEstimates, DeLellisSpadaro2016:AreaMinimizingCurrents2CenterManifold, DeLellisSpadaro2016:AreaMinimizingCurrents3BlowUp}, the implementation in our setting is different.  
Indeed, in the works of Almgren-De Lellis-Spadaro the monotonicity of Almgren's frequency function are obtained as a combination of a Calder\'on-Zygmund type decomposition and a construction of a center manifold, which are intrinsically interconnected and cannot be treated independently. 
However, in our case, the construction of the $m$-hodograph transform from Step 1 and the Calder\'on-Zygmund decomposition from point (iv) are treated independently. We also notice that in \cite{Almgren2000:AlmgrensBigRegularityPaper, DeLellisSpadaro2016:AreaMinimizingCurrents1LpEstimates, DeLellisSpadaro2016:AreaMinimizingCurrents2CenterManifold, DeLellisSpadaro2016:AreaMinimizingCurrents3BlowUp} point (i) of the startegy above is replaced by a Lipschitz approximation argument. Moreover, in the error terms of our frequency function there are two main differences with respect to Almgren-De Lellis-Spadaro; these different error terms are due to the nature of our functional and do not allow us to implement directly their construction. 

\begin{itemize}
\item {\it Third order gradient terms.} The presence of a third order gradient term in the functional \eqref{e:intro-functional-w} is a natural consequence of hodograph-type transformations. In fact, the only functional which does not produce third order gradient terms after hodograph transformations is the area functional. 

\item {\it Boundary terms.} Our frequency function is a boundary frequency function, in the sense that both the energy \eqref{eqn:onePhaseEnergyDef} and the errors (see \cref{lem:freqidentities}) contain terms which are due to the boundary condition and ultimately to the presence of the free boundary.  
\end{itemize}

\medskip

\noindent \textbf{Step 3. Blow-up.} Finally, the estimate on the Hausdorff dimension of the branching set is obtained by contradiction via a frequency blow-up, as a consequence of the analogous estimate for the harmonic thin-obstacle problem. The main ingredients are the almost-monotonicity formula (\cref{thm:freqmon}) for the frequency function from Step 2, and a dimension reduction lemma adapted to our problem (\cref{lem:dimensionreduction}).

\subsection{Unique continuation for quasilinear elliptic operators and variational inequalities}\label{subsection:unique-continuation}

As a byproduct of our analysis, we obtain unique continuation results for a class of quasilinear operators. In particular, \cref{thm:boundaryuniquecontinuation} and \cref{thm:boundaryuniquecontinuation2} below are an immediate consequence of the results from  \cref{section:L2LinftyEstimates}, \cref{section:OuterAndInnerVariations}, \cref{section:ErrorEstimates}, and \cref{section:FrequencyMonotonicity}. An analogous argument would apply to a larger class of operators, but for the sake of simplicity we chose to present it only for the operator associated to the functional from \eqref{e:intro-functional-w}. In fact, using the theory developed in the present paper, in the forthcoming \cite{FerreriSpolaorVelichov2024:UniqueContinuationNonlinearVariational} we are able to prove weak and strong unique continuation results for a general class of nonlinear elliptic problems, which are not covered by the existing unique continuation literature. 
\begin{theorem}[Boundary unique continuation]\label{thm:boundaryuniquecontinuation}
    Let $w\in C^{1,\sfrac12}(\overline{B_1^+})\cap W^{1,2}(B_1^+)$ be a minimizer of the functional $\mathcal F+\mathcal E$, defined in \cref{thm:thinobs}, among all functions with prescribed boundary datum on $(\partial B_1)^+$ that vanish on $B_1'$. Then either $w\equiv 0$ or
    \[
    \dim(\{x=(x',0)\,:\,w(x)=0\,\text{ and }\,|\nabla w|(x)=0 \}\cap B_{\sfrac12})\leq d-2\,.
    \]
\end{theorem}
\begin{theorem}[Dimension of the free boundary in a thin-obstacle problem]\label{thm:boundaryuniquecontinuation2}
    Let $w\in C^{1,\sfrac12}(\overline{B_1^+})\cap W^{1,2}(B_1^+)$ be a minimizer of the functional $\mathcal F+\mathcal E$, defined in \cref{thm:thinobs}, among all functions with prescribed boundary datum on $(\partial B_1)^+$ which are non-negative on $B_1'$. Then either $w\equiv 0$ or 
    \[
    \dim(\{x=(x',0)\,:\,w(x)=0\,\text{ and }\,|\nabla w|(x)=0 \}\cap B_{\sfrac12})\leq d-2\,.
    \]
\end{theorem}

\subsection{On the branching set of the two-phase problem with analytic separation}\label{section:BranchinSetTwoPhaseAnalyticSeparation}

As a corollary of \cref{thm:onePhaseAnalyticObstacle} we obtain an estimate on the dimension of the set of branching points in the two-phase problem, under the following analytic separation property:
\begin{definition}[Analytic separation condition]\label{def:AnalyticSeparationCondition}
        We say that a solution $u$ to the two-phase problem satisfies the analytic separation condition if there exists an analytic function $\phi:B_1'\to\R$ such that
        \[
        \mathrm{graph}(\phi) \subseteq \{ u=0 \} \cap B_1.
        \]
    \end{definition}
We recall that a sign-changing function $u\in H^1(B_1)$ is said to be a solution of the two-phase Bernoulli free boundary problem in $B_1 \subset \R^d$ if 
        \[
        J_2(u,B_1)\le J_2(v,B_1)\quad\text{for every}\quad v\in H^1(B_1)\quad\text{such that}\quad u=v\quad\text{on}\quad \partial B_1,
        \]
        where the two-phase functional $J_2$ is defined as
        \[
        J_2(v,B_1):= \int_{B_1} \vert \nabla v \vert^2\,dx + \lambda_+\vert 
        \{v>0\} \cap B_1 \vert+ \lambda_-\vert 
        \{v<0\} \cap B_1 \vert\,,
        \]
        and where $\lambda_\pm>0$ are fixed constants. It is known from \cite{DePhilippisSpolaorVelichkov2021:TwoPhaseBernoulli} that, in any dimension $d\ge 2$, the free boundaries $\partial\Omega_u^\pm=\partial\{\pm u>0\}$ are two $C^{1,\alpha}$ regular manifolds of dimension $d-1$ in a neighborhood of the contact set 
        $$\mathcal C_2(u):=B_1\cap\partial\Omega_u^+\cap\partial\Omega_u^-.$$
        The set of branching points in this case is defined as 
\begin{equation}\label{e:intro-two-phase-branching-points}
\mathcal B_2(u):=\Big\{x\in \mathcal C_2(u)\ :\ \left|B_r(x)\cap  \{u=0\}\right|\neq 0\ \text{ for every }\ r>0\Big\},
\end{equation}
while the set of singular points $\mathcal S_2(u)$ is given by
\begin{equation}\label{e:def-S-two-phase}
\mathcal S_{2}(u):=\Big\{x\in \mathcal C_2(u)\,:\,|\nabla u_+|(x)=\sqrt{\lambda_+}\ \text{and}\ |\nabla u_-|(x)=\sqrt{\lambda_-}\Big\}.
\end{equation}
As in the boundary one-phase case, we have that $\mathcal B_{2}(u)\subset \mathcal S_{2}(u)$.
    \begin{theorem}[Unique continuation for the two-phase problem with analytic separation]\label{theorem:twoPhaseSymmetricBranchingEst}
        Let $u$ be a solution of the two-phase Bernoulli free boundary problem in $B_1 \subset \R^d$ and suppose that $u$ satisfies the analytic separation condition. Then, either
        \[
        \text{$u$ is harmonic in $B_1$}
        \]
        or
        \[
        \text{the set of branching points of $u$ has Hausdorff dimension at most $d-2$}.
        \]
        Moreover, in the second case, if $d=2$ then the set of branching points is locally finite.
    \end{theorem}
    \begin{proof}
    The claim follows from \cref{thm:onePhaseAnalyticObstacle} since the positive and the negative parts of $u$ are separately solutions to \eqref{e:intro-onePhase-variational-problem}. Indeed, if we take $\mathcal A$ to be epigraph of $\phi$ and if we define the function $\widetilde u:=\lambda_+^{-\sfrac12}u_+$, then $\widetilde u$ is a solution to \eqref{e:intro-onePhase-variational-problem} in $B_1$ and set $\mathcal S_1(\widetilde u)$ is contained in $\mathcal S_2(u)$.
    \end{proof} 

\section{Change of coordinates to a quasilinear boundary unique continuation problem}\label{section:onePhaseFrequencyChangeOfCoordinates}

In this section we construct a suitable change coordinates that will allow us to recast the minimization problem \eqref{e:def-S-one-phase} as a minimization problem for a suitable nonlinear thin-obstacle type energy (see \cref{thm:thinobs}). We start with the following application of Cauchy-Kovalevskaya theorem.

\begin{proposition}[Extension of the obstacle]\label{prop:CK}
    Given an analytic function $\phi$ as in Theorem \ref{thm:onePhaseAnalyticObstacle}, there exist a radius $\rho = \rho(\phi)>0$ and an analytic  function $m\colon B_\rho\to \R $, such that if we set $\Omega_m:=\mathcal A$ then 
    $$\begin{cases}
\Delta m=0\,\text{ and }\,m>0&\text{in}\quad B_\rho\cap \Omega_m,\\
\nu_\phi\cdot\nabla m=-1\,\text{ and }\, m=0&\text{on}\quad B_\rho\cap \partial\Omega_m,
\end{cases}$$
where $\nu_\phi$ is the outward unit normal to $\Omega_m$.  
\end{proposition}

\begin{proof}
    The proof is a straightforward application of Cauchy-Kovalevskaya theorem, noticing that the boundary hypersurface, i.e. ${\rm graph}(\phi)$, and the boundary conditions are all analytic. The positivity of $m$ on $\Omega_m$ follows from the gradient condition.
\end{proof}
Next, using a suitable flow associated to the vector field $\nabla m$, we define a new change of coordinates which sends the level sets of $m$ to the ones of $x_d$.
\begin{proposition}[Change of coordinates]\label{prop:changeofcoord}
Given a $C^2$ function $m$ as in \cref{prop:CK}, there exist $\delta>0$ and a $C^1$ diffeomorphism onto its image $\Phi \colon B_\delta'\times[0,\delta) \to B_\rho$ such that 
$$\begin{cases}
\partial_t\Phi(x',t)=\frac{\nabla m}{|\nabla m|^2}(\Phi(x',t))\quad\text{for every}\quad t\in[0,\delta),\\
\Phi(x',0)=(x',\phi(x'))\in\partial\Omega_m.
\end{cases}$$
In particular, for every $(x',t)\in B_\delta'\times[0,\delta)$, it holds
\begin{equation}\label{e:PhiProperties-levels-of-m}
m(\Phi(x',t))=t;
\end{equation}
\begin{equation}\label{e:PhiProperties-Phi0}
\partial_i\Phi(x',0)\cdot\partial_j\Phi(x',0)=\delta_{ij}+\partial_i\phi(x')\partial_j\phi(x')\quad\text{for every}\quad i,j=1,\dots,d-1;\end{equation}
\begin{equation}\label{e:PhiProperties-orto}
\partial_j\Phi(x',t)\cdot\partial_t\Phi(x',t)=0\quad\text{for every}\quad j=1,\dots,d-1;
\end{equation}
\begin{equation}\label{e:PhiProperties-Delta-m}
\partial_t\Big(|\nabla m|^2(\Phi(x',t))\det(D\Phi(x',t))\Big)=0\,.
\end{equation}
\end{proposition}

\begin{proof} The existence of $\delta$ and $\Phi$ and its regularity are a standard consequence of Cauchy-Lipschitz and the condition $|\nabla m|=1$ on $\partial \Omega_m$.

The property \eqref{e:PhiProperties-levels-of-m} follows directly from the definition of the flow $\Phi$ by differentiating $m(\phi(x',t))$ in $t$, and \eqref{e:PhiProperties-orto} follows by taking the derivative of \eqref{e:PhiProperties-levels-of-m} with respect to $x_j$. Finally, in order to prove \eqref{e:PhiProperties-Delta-m}, we first notice that 
$$(D\Phi)(D\Phi)^{T}=\begin{pmatrix}
    A & 0 \\
    0 & 1/\vert \nabla m \vert^2
\end{pmatrix},$$
where $A=A(x',t)$ is the symmetric $(d-1)\times(d-1)$ matrix with coefficients $A_{ij}=\partial_i\Phi\cdot\partial_j\Phi$, $i,j=1,\dots,d-1$. 
Thus
$$\det(D\Phi)=\sqrt{\det((D\Phi)(D\Phi)^{T})}=\frac1{\vert \nabla m \vert(\Phi)}\sqrt{\det A},$$
For any smooth set $\Omega'\subset B_\delta'\subset \R^{d-1}$, and for any $0\le t<s<\delta$, we set 
$$\Omega_{t,s}=\Big\{\Phi(x',\tau)\ :\ \tau\in(s,t),\ x'\in \Omega'\Big\},$$
and we compute
\begin{align*}
\int_{s}^{t}&\int_{\Omega'}\partial_t\Big[|\nabla m|^2(\Phi)\det(D\Phi)\Big]\\
&=\int_{B_\rho}\int_{s}^{t}\partial_t\Big[|\nabla m|^2(\Phi)\det(D\Phi)\Big]\\
&=\int_{\Omega'}\Big[|\nabla m|^2(\Phi(x',t))\det(D\Phi(x',t))-|\nabla m|^2(\Phi(x',s))\det(D\Phi(x',s))\Big]\\
&=\int_{\Omega'}\Big[|\nabla m|(\Phi(x',t))\sqrt{\det A(x',t)}-|\nabla m|(\Phi(x',s))\sqrt{\det A(x',s)}\Big]\\
&=\int_{\Phi(\Omega',t)}|\nabla m|-\int_{\Phi(\Omega',s)}|\nabla m|+\underbrace{\int_{\{\Phi(x',\tau)\ :\ \tau\in(s,t),\ x'\in \Omega'\}}\nu\cdot \nabla m}_{=0}\\
&=\int_{\partial \Omega_{t,s}}\nu\cdot\nabla m=\int_{ \Omega_{t,s}}\Delta m = 0,
\end{align*}
where $\nu$ is the exterior normal to $\Omega_{t,s}$. Since $\Omega'$, $s$ and $t$ are arbitrary, we get \eqref{e:PhiProperties-Delta-m}.
\end{proof}

We can finally state the main result of this section.

\begin{proposition}[Reduction to a quasilinear boundary unique continuation problem]\label{thm:thinobs}
    Let $u$, $\phi$ and $f$ be as in \cref{thm:onePhaseAnalyticObstacle} and let $0\in \mathcal S_1(u)$. 
Given $m, \Phi$ as in \cref{prop:CK} and \cref{prop:changeofcoord}, and setting  $A:=(\partial_i \Phi \cdot \partial_j\Phi)_{i,j=1}^{d-1}$ and let $S=A^{-1}$ and we define the elliptic, symmetric matrix $M$ with smooth coefficients by
\begin{equation}\label{eq:elliptic_matrix}
M(x):=\det(D\Phi)(x)\, \begin{pmatrix}
    S(x) & 0 \\
    0 & \vert \nabla m \vert^2(\Phi(x))
    \end{pmatrix} \,,
\end{equation}
and the smooth function $Q$ by
 \[
Q(x', x_d) \coloneqq \left( 1 - \vert \nabla m \vert^2 (\Phi(x', x_d)) \right) \det\left( D\Phi \right)(x', x_d).
\]
Then there exists a Lipschitz function $w\colon B_1^+\to \R$ which minimizes
\begin{equation}\label{eqn:thinObstacleRobinRegularity}
\min\left\{ \mathcal{F}(\varphi)+\mathcal E(\varphi) : \varphi \in \mathcal{L}_{\sfrac12} \text{ with } \varphi \ge 0 \text{ on } B_1' \text{ and } \varphi = w  \text{ on } \partial B_1^+ \right\},
\end{equation}
with
\[
\mathcal{L}_{\sfrac12} \coloneqq \left\{ u \in Lip\left( \overline{B_1^+}\right) : \| u \|_{Lip\left( \overline{B_1^+} \right)} \le \sfrac12 \right\}
\]
and where the energies $\mathcal F,\,\mathcal E$ are defined by
\begin{align}
    &\mathcal F(w;r)
        :=\int_{B^+_r} \Big(M(x) \nabla w\cdot \nabla w + \partial_d Q(x)\,w\,\partial_dw\Big),\\
&\mathcal E(w;r):=\int_{B^+_r}\left(\sum_{k=1}^3g_k(x,w,\nabla w)w^{3-k}\,P_k(\nabla w)\right),
\end{align}
where $g_k(x,y,p)$, $k=1,2,3$, are analytic functions of $x,y,p$ and where $P_k(p)$, $k=1,2,3$, are $k$-homogeneous polynomials of $p\in\R^d$. 
Moreover 
\begin{equation}\label{eq:contact_set_x}
(x',f(x'))\in \mathcal C_1(u)\ \Leftrightarrow\ w(x',0)=0,
\end{equation}
\begin{equation}\label{eq:branch_set_w}
(x',f(x'))\in \mathcal S_1(u)\ \Leftrightarrow\ w(x',0)=0\ \text{and}\ |\nabla w|(x',0)=0,
\end{equation}
where $\mathcal C_1(u)$ and $\mathcal S_1(u)$ are given by \eqref{e:def-C-one-phase} and \eqref{e:def-S-one-phase}.
\end{proposition}

\begin{proof}
    Up to rescaling, we can take $\delta=1$ in  \cref{prop:changeofcoord}. Let $v:B_1'\times [0,1] \to \R$ be defined as
\[
v(x', x_d) \coloneqq u \left( \Phi(x', x_d) \right),
\]
where $\Phi$ is the flow from \cref{prop:changeofcoord}.
Then we can rewrite the one-phase Bernoulli functional
\begin{align*}\label{eqn:onePhaseEnergyu}
    J_1(u, &\Phi(B_1'\times[0,1]))
        =\int_{\Omega_u\cap\Phi(B_1'\times[0,1]) } \left(|\nabla u|^2+1\right)\notag\\
        &= \int_{ \{ v > 0 \}\cap(B_1'\times[0,1]) } \left(  S \nabla_{x'} v \cdot \nabla_{x'} v + \vert \nabla m \vert^2(\Phi) \vert \partial_d v \vert^2 + 1 \right) \det\left( D\Phi \right)\,\\
        &=\int_{ \{ v > 0 \}\cap(B_1'\times[0,1]) } \left(  S \nabla_{x'} v \cdot \nabla_{x'} v + \vert \nabla m \vert^2(\Phi) ( \partial_d v -1)^2 + 1 -|\nabla m|^2(\Phi)\right) \det\left( D\Phi \right)\\
        &\qquad -2\underbrace{\int_{\{v>0\}\cap (B_1'\cap[0,1])}\partial_d(|\nabla m|^2(\Phi)|\det(D\Phi)|)\,v}_{=0}\\
        &\qquad\qquad 
    +2\underbrace{\int_{\partial (B_1'\cap[0,1])\cap\{v>0\}} |\nabla m|^2(\Phi)\,v\,|\det(D\Phi)|}_{=b}\,,
\end{align*}
where $b$ is a constant depending only on the values of $v$ on $\partial (B_1'\cap[0,1])$,
and where we used the relation
\[
\vert \nabla  u \vert^2(\Phi(x',x_d)) = S(x',x_d)\, \nabla_{x'} v(x',x_d) \cdot \nabla_{x'} v(x',x_d) + \vert \nabla m \vert^2 \left( \Phi(x', x_d) \right) \vert \partial_d v(x',x_d) \vert^2\,.
\]
Now let us consider the hodograph transformation
\begin{equation}\label{eqn:changeCoordHodoTrsnsformDef}
T:B_1'\times[0,1]\to\R^{d}\ ,\qquad T(x',x_d)=(x', v(x', x_d)). 
\end{equation}
Notice that we have 
\begin{equation}\label{eq:nablas}\partial_dv=\nabla u(\Phi)\cdot \partial_d\Phi=\nabla u(\Phi)\cdot \frac{\nabla m}{|\nabla m|^2}(\Phi),
\end{equation}
therefore, since $0$ is a branching point, we have that $\partial_dv(0,0)=1$, so, by the $C^{1,\alpha}$ regularity of $u,m$, for some $\delta$ small enough, the map $T:B_\delta'\times[0,\delta]\to T(B_\delta'\times[0,\delta])$ is invertible. Once again, up to rescaling, we can assume that $\delta=1$. The inverse of $T$ is of the form
$$T^{-1}(x',x_d)=(x',\widetilde w(x',x_d)),$$
where $\widetilde w:T(B_1'\times[0,1])\to\R$ is a $C^{1,\alpha}$ function, which satisfies the following identity
\[
v \left( x', \widetilde w (x', x_d) \right) = x_d,
\]
so that, for any $(x',x_d)\in T(B_1'\times[0,1])$, we have
\begin{align*}
    \begin{split}
        & \nabla_{x'} v\left( x', \widetilde w (x', x_d) \right) + \partial_{x_d} v\left( x', \widetilde w (x', x_d) \right) \,\nabla_{x'} \widetilde w(x', x_d) = 0, \\
        & \partial_{x_d} v\left( x', \widetilde w (x', x_d) \right) \,\partial_{x_d} \widetilde w(x', x_d) = 1.
    \end{split}
\end{align*}
Let us consider the linearization $w$ of $\widetilde w$, defined by the relation
\[
\widetilde w = x_d + w.
\]
For any set $\Omega\subset T(B_1'\times[0,1])$, we can compute the one-phase  energy $J_1(u,\Phi(T^{-1}(\Omega)))$ in terms of the function $w$, by using the identities above and the fact that $\det(T^{-1})=1+\partial_dw$.
\begin{align*}
    J_1&(u, \Phi(T^{-1}(\Omega)))-b\\
       &= \int_{\Omega} \frac{\left(  S(x', x_d+w) \nabla_{x'} w \cdot \nabla_{x'} w + \vert \nabla m \vert^2 (\Phi(x', x_d+w)) |\partial_d w|^2  \right)}{1 + \partial_d w} \det\left( D\Phi (x', x_d+w)\right) \\
       &\quad +\int_\Omega (1-|\nabla m|^2(\Phi(x',x_d+w)))\, |\det(D\Phi(x',x_d+w))|\,(1+\partial_dw)\,=: \mathcal{F}_h + \mathcal{F}_b.
\end{align*}
It is immediate to check that $\mathcal F_h$ can be written in the form
\begin{align*}
\mathcal{F}_h &= \int_{\Omega} \left(  S(x', x_d+w) \nabla_{x'} w \cdot \nabla_{x'} w + \vert \nabla m \vert^2 (\Phi(x', x_d+w)) \vert \partial_d w \vert^2 \right) \frac{\det\left( D\Phi (x', x_d+w)\right) }{1 + \partial_d w}\\
&= \int_{\Omega} \Big(M(x) \nabla w \cdot \nabla w+\sum_{k=2}^3\widetilde g_k(x,w,\nabla w) \,\widetilde P_k(\nabla w)w^{3-k}\Big),
\end{align*}
where, for every $k=2,3$, $\widetilde g_k=\widetilde g_k(x,y,p)$ is an analytic function of $x\in\R^d,y\in\R$ and $p\in\R^d$ and where $\widetilde P_k(p)$ is $k$-homogeneous polynomial of $p\in\R^d$.

Next observe that 
\begin{align*}
    \mathcal F_b&=\int_\Omega \left(Q(x)+\partial_{d} Q(x) \,w+g(x,w)w^2\right)\,(1+\partial_dw)\\
    &=\int_\Omega Q(x)+\partial_d(Q(x)\,w)+\partial_dQ(x)\,w\,\partial_dw+g_2(x,w,\nabla w)\,w\,P_2(\nabla w)+g_1(x,w,\nabla w)\,w^2\,P_1(\nabla w)\,\\
    &=\int_\Omega Q(x)+\partial_d(|\det D\Phi(x)|)\,w\,\partial_dw+g_2(x,w,\nabla w)\,w\,P_2(\nabla w)+g_1(x,w,\nabla w)\,w^2\,P_1(\nabla w)\\
    &\quad -\underbrace{\int_{B_r'\cap \Omega}Q(x)\,w}_{=0}+b\,,
\end{align*}
where in the last equality we used that $Q(x)\equiv 0$ on $B_r'$ and $b$ is a number depending only on the value of $w$ on $\partial B_r$.

Conversely, if $\varphi$ is a Lipschitz function in $T(B_1'\cap[0,1])$, which is such that $|\nabla \varphi|\le \sfrac12$ in $\Omega$, $\varphi\equiv w$ outside of $\Omega$ and $\varphi\ge 0$ on $\Omega\cap \{x_d=0\}$, then, for every fixed $x'$, the function
$$x_d\mapsto x_d+\varphi(x',x_d)$$
is monotone increasing and invertible on the segment $\{(x',t)\ :\ (x',t)\in T(B_1'\cap[0,1])\}$.  
The map $\widetilde T^{-1}:T(B_1'\times[0,1])\to\R^d$ defined as
$$\widetilde T^{-1}(x',x_d):=(x',x_d+\varphi(x',x_d))\,$$
has an image $Im(\widetilde T^{-1})\subset B_1^+\cap[0,1]$. The map $\widetilde T^{-1}:T(B_1'\times[0,1])\to Im(T^{-1})$ is a bi-Lipschitz homeomorphism and the symmetric difference $Im(T^{-1})\Delta(B_1'\times[0,1])$ is contained in a neighborhood of the origin. Thus, the inverse $\widetilde T$ of $\widetilde T^{-1}$ is of the form $\widetilde T(x',x_d)=(x',\widetilde \phi(x',x_d))$, where $\widetilde \phi$ is a competitor for $v$ in a neighborhood of the origin. Since $v$ solves \eqref{e:intro-onePhase-variational-problem}, the minimality of $w$ follows.

Next notice that $(x',f(x'))\in \mathcal C_1(u)$ if and only if $0=u(x',f(x'))=m(x', f(x'))$ if and only if $\Phi(x',0)=(x',f(x'))$ and $$w(x',0)=v(x',0)=m(\Phi(x',0))=0\,,$$ where in the last equality we used \eqref{e:PhiProperties-levels-of-m}. Finally, $(x',f(x'))\in \mathcal S_1(u)$ if and only if $(x',f(x'))\in \mathcal C_1(u)$ and $|\nabla u(x',f(x'))|=1$ if and only if $\Phi(x',0)=(x',f(x'))$, $w(x',0)=0$ and 
\[
\partial_d w(x',0)=\partial_d v(x',0)-1=\nabla u(x',f(x'))\cdot \frac{\nabla m(x',f(x'))}{|\nabla m|^2}(x',f(x'))-1=0\,,
\]
where in the last equality we used that $\nabla u(x',f(x'))=\nabla m(x',f(x'))$ and \eqref{eq:nablas}.
\end{proof}
To conclude this section, we notice that from the construction in the proof of \cref{thm:thinobs} we have the following
\begin{corollary}[Equivalence of the $C^{1, \alpha}$ estimates]\label{corollary:C1auwEquivalence}
Let $\alpha \in (0, \sfrac12)$, the functions $u$, $m$, $\phi$, $f$ and $w$ be as in Theorem \ref{thm:onePhaseAnalyticObstacle} and let $0\in \mathcal S_1(u)$. Then, there exist constants $\delta_0= \delta_0(d) > 0$ and $c = c(d,\delta_0)>0$ such that, if $\delta\in(0,\delta_0)$ and 
\[
\|u - x_d\|_{C^{1, \alpha}\left( \Omega_u \cap B_1 \right)} + \|f\|_{C^{1,\alpha}(B'_1)} \le \delta \quad \text{and} \quad \|m - x_d\|_{C^{3, \alpha}\left( B_1 \right)} + \|\phi\|_{C^{3,\alpha}(B'_1)} \le \delta
\]
then the following implications hold:
\begin{itemize}
    \item[(i)] 
    \[
    \| u - m \|_{C^{1, \alpha}\left( \Omega_u \cap B_1 \right)} + \| f - \phi \|_{C^{1, \alpha}\left( B_1' \right)} \le c \delta \implies \| w \|_{C^{1, \alpha}\left( T\left( \Omega_u \cap B_1 \right) \right)} \le \delta,
    \]

    \item[(ii)]
    \[
    \| w \|_{C^{1, \alpha}\left( T\left( \Omega_u \cap B_1 \right) \right)} \le c \delta \implies \| u - m \|_{C^{1, \alpha}\left( \Omega_u \cap B_1 \right)} + \| f - \phi \|_{C^{1, \alpha}\left( B_1' \right)} \le \delta,
    \]
\end{itemize}
where the function $T$ is the hodograph transformation defined in \eqref{eqn:changeCoordHodoTrsnsformDef}.
\end{corollary}

\section{$L^\infty$-$L^2$ estimates and unique continuation}\label{section:L2LinftyEstimates}
In this section we will consider minimizers of $\mathcal G(\cdot, 1):=\mathcal F(\cdot, 1)+\mathcal E(\cdot, 1)$ and of the associated linear energy $\mathcal F(\cdot, 1)$ in the class of competitors
\[
\mathcal A:=\left\{\varphi \in \mathcal{L}_{\sfrac12} \,:\, \varphi \ge 0 \text{ on } B_1' \text{ and } \varphi = w  \text{ on } \partial B_1^+ \right\}
\]
In the sequel we will omit the $1$ in the energies. In particular we will be interested in proving a $C^{1}$-$L^2$ estimate for minimizers and moreover we will need a suitable unique continuation type property.

The aim of this section is to adapt the existing $C^{1, \alpha}$ regularity theory for nonlinear thin obstacle problems of the form \eqref{eqn:onePhaseNonlinearThinObstacleVariationalIneq} to obtain a quantitative estimate, and to provide a useful quantitative unique continuation property for such solutions under a doubling assumption. 

Given \eqref{corollary:C1auwEquivalence}, in the following of this paper it is not restrictive to work under the following
\begin{assumptions}[A priori $C^{1, \alpha}$ estimate]\label{assumptions:wC1aAprioriEst}
    Let the constants $\alpha$ and $\delta_0$ as in \cref{corollary:C1auwEquivalence}. Moreover, we suppose that the function $w:B_1^+ \to \R$ in non-trivial and satisfies 
    \[
    \| w \|_{C^{1, \alpha}\left( \overline{B_1^+} \right)} \le \delta,
    \]
    for some $\delta\in(0,\delta_0)$.
\end{assumptions}
Moreover, by a rotation of the coordinate system, we can also make use of the following
\begin{assumptions}[Branching condition]\label{assumptions:startingAssumptions}
    The origin is a contact point of the free boundaries and $\nabla m(0) = e_d$. In particular,
    \[
    w(0) = 0 \quad \text{and} \quad \nabla w(0) = 0.
    \]
\end{assumptions}
\begin{remark}[Non-degeneracy]\label{remark:wNonDegeneracy}
    Let the function $w:B_1 \to \R$ be a solution of \eqref{eqn:thinObstacleRobinRegularity} under \cref{assumptions:wC1aAprioriEst} and \cref{assumptions:startingAssumptions}. Then, 
    \begin{equation*}
    \int_{B_r^+} w^2 > 0 \quad \text{and} \quad \int_{B_r^+} \vert \nabla w \vert^2 > 0 \quad \text{for all } r \in (0, 1).
    \end{equation*}
    Indeed, suppose that either one fails. Then, $w \equiv 0$ in $B_r^+$, since $w(0) = 0$. On the other hand, \cref{corollary:C1auwEquivalence} would imply
    \[
    \| u - m \|_{C^{1, \alpha}\left( \Omega_u \cap B_{r_0} \right)} + \| f - \phi \|_{C^{1, \alpha}\left( B_{r_0}' \right)} = 0
    \]
    for some radius $r_0>0$ sufficiently small. However, as a consequence of the unique continuation property for harmonic functions
    \[
    \partial \Omega_u \cap B_1 \equiv \mathrm{graph}(\phi) \cap B_1 \quad \text{and} \quad u \equiv m \text{ in } \Omega _u \cap B_1, 
    \]
    which is in contradiction with the non-triviality of $w$.
\end{remark}

\subsection{$L^\infty$-$L^2$ estimates}
We have the following result, whose proof is postponed to \cref{appendix:C1aRegularityNonlinearSignorini}.
\begin{proposition}[$C^{1, \alpha}$ estimate]\label{prop:nonlinearThinObstacleC1aRegularity}
    Let the function $w:B_1 \to \R$ be a solution of \eqref{eqn:thinObstacleRobinRegularity} under \cref{assumptions:wC1aAprioriEst}. Then, there exists a constant $C = C(d, \delta_0)>0$ such that
    \[
    \| w \|_{C^{0, \alpha}\left( B_{r/2}(x) \cap B_1^+ \right)} \le C \left( \fint_{B_r(x) \cap B_1^+} w^2 \right) \quad \text{and} \quad \| \nabla w \|_{C^{0, \alpha}\left( B_{r/2}(x) \cap B_1^+ \right)} \le \frac{C}{r} \left( \fint_{B_r(x) \cap B_1^+} w^2 \right)
    \]
    for all balls $B_r(x)$ with $B_r(x)^+ \subset B_1^+$.
\end{proposition}
There are several contributions in the literature concerning the regularity of solutions for nonlinear thin obstacle problems. Concerning the H\"{o}lder regularity we refer for instance to \cite{BeiraoDaVeigaConti1972:ThinObstacleC0a}, the Lipschitz continuity was first addressed in \cite{GiaquintaModica1975:LipschitzRegularityNonlinearObstacles}, while the continuity of the first derivatives was first proved in \cite{Frehse1977:NonlinearThinObstacleC1}. The $C^{1, \alpha}$ regularity in the nonlinear setting, on the other hand, is much more recent and has been proved in \cite{DiFazioSpadaro2022:NonlinearThinObstacleC1a}.

For our purpose, since we already have an a priori $C^{1, \alpha}$ estimate arising from the improvement of flatness for the Bernoulli problem, it will suffice to adapt the linearization method developed in \cite{RulandShi:C1aThinObstacleBlowUp} for (linear) thin obstacle problems with H\"{o}lder continuous coefficients. We refer to \cref{appendix:C1aRegularityNonlinearSignorini} for the details of the proof.

\subsection{Localized unique continuation}

In order to prove the monotonicity of the frequency function we will need to estimate quantities at the same scale (this is the key point in the proof of \cref{lemma:energyEstCubesHD} below).  This will be possible thanks to the following unique continuation type estimates for solutions of \eqref{eqn:thinObstacleRobinRegularity} under a doubling assumptions. 
\begin{lemma}[Three annuli-type lemma I]\label{lemma:excessOnePhaseThreeAnnuliLemma}
For every constants $\eta \in (0, 1)$ and $C > 0$, there exists $\gamma = \gamma(d, \eta, C, \delta_0) > 0$ and $r_0 = r_0(d, \eta, C, \delta_0) > 0$ with the following property.
Suppose that the function $w$ is a solution of \eqref{eqn:thinObstacleRobinRegularity} under \cref{assumptions:wC1aAprioriEst} and that
\begin{equation}\label{eqn:excessDecayPropertyThreeAnnuli}
    \int_{B_{3r}(x) \cap B_1^+}|\nabla w|^2 \le C \int_{B_{r}(x) \cap B_1^+}|\nabla w|^2,
\end{equation}
\begin{equation}\label{eqn:heightBoundThreeAnnuli}
    \int_{B_{3r}(x) \cap B_1^+} w^2\leq C r^2 \int_{B_{r}(x) \cap B_1^+} |\nabla w|^2,
\end{equation}
for some $r<r_0$ and $x \in \overline{B_{1/3}^+}$. 
Then, for all $y \in B_1^+$ such that $B_{\eta r}(y) \subseteq B_{r}(x)$
\begin{equation}\label{eqn:excess3AnnuliLemmaCubes}
\int_{B_{\eta r}(y) \cap B_1^+}|\nabla w|^2 \ge \gamma \int_{B_{3 r}(x) \cap B_1^+}|\nabla w|^2.
\end{equation}
\end{lemma} 
\begin{proof}
    By contradiction, suppose \eqref{eqn:excess3AnnuliLemmaCubes} fails for a sequence of constants $\gamma_n \to 0$, minimizers $w_n$ and balls $B_{r_n}(x_n)$ and $B_{\eta r_n}(y_n)$, with radii $r_n \to 0$ and centers $x_n$ and $y_n$ respectively. Up to a subsequence, we can assume that $x_n \to x_{\infty}$. We distinguish two cases:
    \begin{enumerate}
        \item $x_{\infty} \in \{ x_d = 0 \}$. Considering the functions
    \[
    \widetilde{w}_n \coloneqq \frac{1}{r_n} w_n\left( r_n \left(x - x_n \right) \right) \text{ on } B_3^+
    \]
    we can rewrite \eqref{eqn:excessDecayPropertyThreeAnnuli} and \eqref{eqn:heightBoundThreeAnnuli} respectively as
    \begin{equation}\label{eqn:rescaledEstimatesThreeAnnuli}
     \int_{B_3} \vert \nabla \widetilde{w}_n \vert^2 \le C  \int_{B_1} \vert \nabla \widetilde{w}_n \vert^2 \quad \text{and} \quad \int_{B_3} \widetilde{w}_n^2 \le C  \int_{B_1} \vert \nabla \widetilde{w}_n \vert^2.
    \end{equation}
    Now, consider their normalizations
    \[
    w'_n \coloneqq \frac{\widetilde{w}_n}{\int_{B_3} \vert \nabla \widetilde{w}_n \vert^2}.
    \]
    As a consequence of the estimates \eqref{eqn:rescaledEstimatesThreeAnnuli}, we can apply the $C^{1, \alpha}$ estimate from \cref{prop:nonlinearThinObstacleC1aRegularity} to deduce the bound
    \[
    \| w'_n \|_{C^{1, \alpha}\left( \overline{B_{3/2}^+} \right)} \le C
    \]
    uniformly in $n$. Hence, the functions $w'_n$ converge weakly in $H^1\left(B_{3/2}^+\right)$ and strongly in $C^{1, \alpha}\left(\overline{B_{3/2}^+}\right)$ to a limit function $w'_{\infty}$, which satisfies:
    \begin{enumerate}
        \item[(i)] $w'_{\infty}$ is an $H^1$ solution of a constant coefficients thin obstacle problem in $B_{3/2}^+$;

        \item[(ii)] $w'_{\infty}$ is nontrivial, since
        \[
        \int_{B_1^+} \vert \nabla w'_{\infty} \vert^2 = 1;
        \]

        \item[(iii)] by the contradiction argument
        \[
        \int_{B_{\eta}(y_{\infty})^+} \vert \nabla w'_{\infty} \vert^2 = 0.
        \]
        where $y_{\infty}$ is the limit of the points $y_n$ in the blow-up scale $r_n$.
        \end{enumerate}
        From (iii) we have that $w'_{\infty}$ is constant on $B_{\eta}(y_{\infty})^+$, and by (i) it is constant on the whole $B_{3/2}^+$, which is in contradiction with (ii).

    \item $x_{\infty} \notin \{ x_d = 0 \}$. In this case one can proceed analogously as in the previous one, with the only difference that this time the centers $x_n$ do not approach $\{ x_d = 0 \}$ so that the limit $w'_{\infty}$ is harmonic in the whole ball $B_3$.\qedhere
    \end{enumerate}
\end{proof}
Now we introduce a lemma similar to \cref{lemma:excessOnePhaseThreeAnnuliLemma}, but with respect to the $L^2$ norm of the function $w$.
\begin{lemma}[Three annuli-type lemma II]\label{lemma:heightOnePhaseThreeAnnuliLemma}
Let the constants $\eta, C, \gamma, r_0$ be as in \cref{lemma:excessOnePhaseThreeAnnuliLemma}. Suppose that the function $w$ is a solution of \eqref{eqn:thinObstacleRobinRegularity} under \cref{assumptions:wC1aAprioriEst} and that
\begin{equation*}
    r^2 \int_{B_{3r}(x) \cap B_1^+}|\nabla w|^2 \le C \int_{B_{r}(x) \cap B_1^+}w^2,
\end{equation*}
\begin{equation*}
    \int_{B_{3r}(x) \cap B_1^+} w^2\leq C \int_{B_{r}(x) \cap B_1^+} w^2,
\end{equation*}
for some $r<r_0$ and $x \in \overline{B_{1/3}^+}$. 
Then, for all $y \in B_1^+$ such that $B_{\eta r}(y) \subseteq B_{r}(x)$
\begin{equation*}\label{eqn:height3AnnuliLemmaCubes}
\int_{B_{\eta r}(y) \cap B_1^+}w^2 \ge \gamma \int_{B_{3 r}(x) \cap B_1^+}w^2.
\end{equation*}
\end{lemma}
\begin{proof}
    The proof is analogous to that of \cref{lemma:excessOnePhaseThreeAnnuliLemma}. The only difference is that in this case the functions $w'_n$ are defined as
    \[
    w'_n \coloneqq \frac{\widetilde{w}_n}{\int_{B_3} \widetilde{w}_n^2},
    \]
    where $\widetilde{w}_n$ are as in the proof of \cref{lemma:excessOnePhaseThreeAnnuliLemma}.
\end{proof}

\section{Outer and inner variations}\label{section:OuterAndInnerVariations}

In this section we study outer and inner variations for minimizers as in \cref{thm:thinobs}. We will use them later to prove monotonicity of a suitable frequency function.

We proceed similarly as in \cite{DeLellisSpadaro2016:AreaMinimizingCurrents3BlowUp, FocardiGelliSpadaro2015:UniqueContinuationLipschitz}. Let $M(x)$ be a uniformly elliptic and symmetric matrix field with Lipschitz continuous coefficients, and $\varphi:\R^+ \to \R^+$ be the Lipschitz cutoff function 
\begin{equation}\label{eqn:onePhaseCutoffDef}
\varphi(x) \coloneqq
\begin{cases}
    1 & \text{if } x \in (0, 1-\upsilon], \\
    \frac{1-x}{1-\upsilon}  & \text{if } x \in (1-\upsilon, 1], \\
    0 & \text{if } x \in (1, +\infty),
\end{cases}
\end{equation}
for some $\upsilon \in (\sfrac12, 1)$ and we set
\begin{equation}\label{eqn:psiCutoffBr}
\psi_r:\R^d\to\R^+\ ,\qquad \psi_r(x) \coloneqq \varphi\left( \frac{\vert x \vert}{r} \right).
\end{equation}
We define the height function
\begin{equation}\label{eqn:onePhaseFrequencyHeightDef}
H(r) \coloneqq - \int_{\R^d_+} \varphi'\left( \frac{\vert x \vert }{r} \right) \mu(x) \frac{w(x)^2}{\vert x \vert}\,,
\end{equation}
where the function $\mu(x)$ is defined as:
\begin{equation}\label{e:definition-of-mu}
\mu(x) \coloneqq \frac{x}{\vert x \vert} \cdot M(x) \frac{x}{\vert x \vert}\,.
\end{equation}
We define the energy function $D(r)$ as 
\begin{equation}\label{eqn:onePhaseEnergyDef}
D(r):=D_i(r)+D_b(r),   
\end{equation}
where 
\begin{align}
D_i(r) &\coloneqq \int_{\R^d_+} \varphi\left( \frac{\vert x \vert }{r} \right) M(x) \nabla w \cdot \nabla w, \label{e:definition-of-D-interior}\\
D_b(r) &\coloneqq -\frac12 \int_{\R^{d-1}\times\{0\}}\varphi\left(\frac{\vert x \vert}{r} \right)\,\partial_d Q(x)\,w^2,\label{e:definition-of-D-boundary}
\end{align}
and combining them we can define the frequency function
\begin{equation}\label{eqn:onePhaseFrequencyN(r)Def}
    N(r) \coloneqq \frac{r D(r)}{H(r)} .
\end{equation}
For later use, let us also define the quantity
\begin{equation}\label{eqn:onePhaseFrequencyL2Def}
G(r) \coloneqq \int_{\R^d_+} \varphi\left( \frac{\vert x \vert }{r} \right) w^2\,.
\end{equation}
Moreover, we introduce the following notation:
\begin{align}
& A(r) \coloneqq - \int_{\R^d_+} \varphi'\left( \frac{\vert x \vert }{r} \right) \frac{\vert x \vert}{\mu(x)} \left( M(x) \nabla w \cdot \frac{x}{\vert x \vert}\right)^2, \label{e:definition-of-A-levico}\\
& B(r) \coloneqq - \int_{\R^d_+} \varphi'\left( \frac{\vert x \vert }{r} \right) w \, \nabla w \cdot M(x) \frac{x}{\vert x \vert},\label{e:definition-of-B-levico}
\end{align}
In the next three subsections we will prove the following lemma.
\begin{lemma}[Frequency identities]\label{lem:freqidentities} Let $H$ be the height function from \eqref{eqn:onePhaseFrequencyHeightDef}. For all $r \in (0, 1)$ either $H(r) > 0$ or $H(r) \equiv 0$. Moreover, if the first case occurs, we have
\begin{equation}\label{eqn:onePhaseHDerivative}
H'(r) - \frac{d-1}{r} H(r) - \frac{2}{r} B(r) - e_{H}(r) = 0.
\end{equation}
\begin{equation}\label{eqn:onePhaseOuterVariationD}
D(r) - \frac{1}{r} B(r)= e_O(r),
\end{equation}
\begin{equation}\label{eqn:onePhaseDDerivative}
(d-2) D(r) - r D'(r) + \frac{2}{r} A(r)
+ e_{I}(r)  = 0.
\end{equation}
where the errors $e_H, e_O, e_I$ have the following estimates
\begin{align}
    |e_{H}(r)| &\le C\,H(r)\label{eqn:onePhaseHeightDerivativeErrorEst}\\
|e_O(r)|&\le C \, G(r)+ \sum_{k=1}^2 E^{o,k}(r),\label{e:error-innervariation-levico}\\
|e_I(r)|&\le C\, \int_{B_r'} \varphi\left(\frac{|x|}r\right)\,w^2
            +\sum_{k=1}^3 E^{i, k}(r).\label{eq:innervar-error-levico}
\end{align}
for some constant $C$ depending only on the $C^3$ norm of $m$, and moreover
\begin{align}   
& |E^{o, 1}(r)|\leq C  \int_{B_r^+} \varphi\left( \frac{\vert x \vert}{r} \right) \left(|w|\, \vert \nabla w \vert^2 + \vert \nabla w \vert^3\right),\label{eq:bulkouter1} \\
& |E^{o, 2}(r)| \leq \frac{C}{r} \int_{ B_r^+ } \left|\varphi'\right|\left( \frac{\vert x \vert}{r} \right)\,\left(|w|^2\,|\nabla w|+|w|\,|\nabla w|^2 \right).\label{eq:bulkouter2}
\end{align}
\begin{align}
|E^{i, 1}(r)| &\le C r\, \int_{B_r^+} \varphi\left( \frac{\vert x \vert}{r} \right) \left(|\nabla w|^2+|w|\,|\nabla w|\right), \label{eq:onePhaseE{i,1}}\\ 
|E^{i, 2}(r)| &\le C \, \int_{B_r^+} \varphi\left( \frac{\vert x \vert}{r} \right) \left(|\nabla w|^3+|w|\,|\nabla w|^2+|w|^2\,|\nabla w|\right)\,, \label{eq:onePhaseE{i,2}}\\
|E^{i, 3}(r)| & \le  \frac{C}{r} \int_{B_r^+} |\varphi'|\left( \frac{\vert x \vert}{r} \right) \vert x \vert \left(|\nabla w|^3+|w|\,|\nabla w|^2+|w|^2\,|\nabla w|\right)\,. \label{eq:onePhaseE{i,3}} 
\end{align}
\end{lemma}

\subsection{Height derivative: proof of  \eqref{eqn:onePhaseHDerivative}}
 If $H(r) = 0$ for some $r \in (0, 1)$, then going back to the original coordinates we can find a (small) ball $B_{\varepsilon}(y_0)$ such that
    \[
    B_{\varepsilon}(y_0) \subset \{ \tilde u > 0\} \cap \{ m>0 \} \quad \text{and} \quad \tilde u = m \text{ on } B_{\varepsilon}(y_0),
    \]
    which implies that $\tilde u \equiv m$ in $B_1$ since they are harmonic on their positivity set. Hence $w \equiv 0$.
    
    Now let us turn to the height derivative. Similarly to \cite[Proposition 3.6]{FocardiGelliSpadaro2015:UniqueContinuationLipschitz}), we write $H$ as follows:
    \begin{align*}
        H(r)&=- \int_{\R^d_+} \varphi'\left( \frac{\vert x \vert }{r} \right) \mu(x) \frac{w^2(x)}{\vert x \vert}\,dx\\
        &=-\int_{\R^d_+} \varphi'\left( \frac{\vert x \vert }{r} \right) \left(\frac{x}{|x|}\cdot M(x)\frac{x}{|x|} \right)\frac{w^2(x)}{\vert x \vert}\,dx\\
        &=-\int_{\R^d_+} r\nabla\left[\varphi\left( \frac{\vert x \vert }{r} \right)\right] \cdot M(x)\frac{x}{|x|} \frac{w^2(x)}{\vert x \vert}\,dx\\
         &=\int_{\R^d_+} r\varphi\left( \frac{\vert x \vert }{r} \right) {\rm div}\left( M(x)\frac{x}{|x|} \frac{w^2(x)}{\vert x \vert}\right)\,dx.
    \end{align*} 
    Therefore differentiating in $r$, we get
    \begin{align*}
    H'(r)
    &=\frac{H(r)}{r}-\frac1r \int_{\R^{d}_+}|x|\varphi'\left( \frac{\vert x \vert }{r} \right) {\rm div}\left( M(x)\frac{x}{|x|} \frac{w^2(x)}{\vert x \vert}\right)\\
     &=\frac{H(r)}{r}-\frac2r \int_{\R^{d}_+}\varphi'\left( \frac{\vert x \vert }{r} \right) w\nabla w\cdot M(x)\frac{x}{|x|}-\frac1r \int_{\R^{d}_+}|x|\varphi'\left( \frac{\vert x \vert }{r} \right) w^2(x){\rm div}\left( M(x)\frac{x}{|x|^2}\right)
    \end{align*}
    The conclusion follows from
     \begin{align*}
    {\rm div}\left( M(x)\frac{x}{|x|^2}\right)&=\partial_i\left( M_{ij}(x)\frac{x_j}{|x|^2}\right)\\
    &=\partial_i M_{ij}(x)\frac{x_j}{|x|^2}+M_{jj}(x)\frac{1}{|x|^2}-2M_{ij}(x)\frac{x_jx_i}{|x|^4}\\
    &=M_{ij,i}(x)\frac{x_j}{|x|^2}+\frac{{\rm tr} (M(x))}{|x|^2}-2\frac{\mu(x)}{|x|^2}\\
    &=M_{ij,i}(x)\frac{x_j}{|x|^2}+\frac{d}{|x|^2}+\frac{{\rm tr} (M(x)-Id)}{|x|^2}-2\frac{\mu(x)}{|x|^2}\\
        &=M_{ij,i}(x)\frac{x_j}{|x|^2}+(d-2)\frac{\mu(x)}{|x|^2}+\frac{{\rm tr} (M(x)-Id)}{|x|^2}-d\frac{\mu(x)-1}{|x|^2}.
    \end{align*}
Setting 
$$e_H:=-\frac1r \int_{\R^{d}_+}|x|\varphi'\left( \frac{\vert x \vert }{r} \right) w^2(x)\left(M_{ij,i}(x)\frac{x_j}{|x|^2}+\frac{{\rm tr} (M(x)-Id)}{|x|^2}-d\frac{\mu(x)-1}{|x|^2}\right),$$
and recalling that $M$ is Lipschitz, we get \eqref{eqn:onePhaseHeightDerivativeErrorEst}.

\subsection{Outer variation: proof of \eqref{eqn:onePhaseOuterVariationD}}

A direct computation yields
\begin{align*}
    0&=\frac{d}{dt}\Big|_{t=0}\Big(\mathcal F(w+t\,\psi_r\,w;r)+\mathcal E(w+t\,\psi_r\,w;r)\Big)\\
    &\ =2 \int_{ B_r^+ } \varphi\left( \frac{\vert x \vert}{r} \right) M(x) \nabla w \cdot \nabla w + \frac{2}{r} \int_{ B_r^+ } \varphi'\left( \frac{\vert x \vert}{r} \right) w M(x) \nabla w \cdot \frac{x}{\vert x \vert} \\
    &\qquad +\int_{ B_r^+ }  \partial_dQ(x)\,\partial_dw \,\varphi\left( \frac{\vert x \vert}{r} \right)w + \partial_dQ \,w\,\partial_d\left(\varphi\left( \frac{\vert x \vert}{r} \right)w \right) +\int \sum_{k=1}^2E^{o, k}(r) \,.
\end{align*}
Integrating by parts and using the identity
\[
\partial_d\left(\partial_d Q\,\varphi\left(\frac{\vert x \vert}{r} \right)\,w^2\right)=\partial_dQ(x)\,\partial_dw \,\varphi\left( \frac{\vert x \vert}{r} \right)w + \partial_dQ \,w\,\partial_d\left(\varphi\left( \frac{\vert x \vert}{r} \right)w \right)+\partial_d^2Q\,\varphi\left( \frac{\vert x \vert}{r} \right)w^2,
\]
we get
\begin{align*}
   0&=2 \int_{ B_r^+ } \varphi\left( \frac{\vert x \vert}{r} \right) M(x) \nabla w \cdot \nabla w + \frac{2}{r} \int_{ B_r^+ } \varphi'\left( \frac{\vert x \vert}{r} \right) w M(x) \nabla w \cdot \frac{x}{\vert x \vert} \\
    &\qquad -\int_{B'_r}\varphi\left(\frac{\vert x \vert}{r} \right)\,\partial_d Q(x)\,w^2-\int_{B_r^+}\varphi\left(\frac{\vert x \vert}{r} \right)\,\partial_d^2 Q(x)\,w^2+\sum_{k=1}^2E^{o, k}(r) \,.
\end{align*}
The above identity reads as:
\begin{align}
D(r)
&=\frac1rB(r)+\frac12\int_{B_r^+}\varphi\left(\frac{\vert x \vert}{r} \right)\,\partial_d^2 Q(x)\,w^2-\frac12\sum_{k=1}^2E^{o, k}(r),\label{e:innervariation-in-DAB-levico-0}
\end{align}
which can be written in the form 
\begin{align}
D(r)
&=\frac1rB(r)+e_O(r),\label{e:innervariation-in-DAB-levico}
\end{align}
by setting
\begin{align*}
e_O(r):=\frac12\int_{B_r^+}\varphi\left(\frac{\vert x \vert}{r} \right)\,\partial_d^2 Q(x)\,w^2-\frac12\sum_{k=1}^2E^{o, k}(r).
\end{align*}

\subsection{Inner variation: proof of \eqref{eqn:onePhaseDDerivative}}
Consider a family of diffeomorphisms of $B_r$ onto itself of the form
\[
T_{\varepsilon}(x) \coloneqq x + \varepsilon \psi_r(x) F(x),
\]
with $\psi_r$ as in \eqref{eqn:psiCutoffBr} and, as in \cite{FocardiGelliSpadaro2015:UniqueContinuationLipschitz}, we choose the vector field $F$ to be
\begin{equation}\label{eqn:InnerVariationFieldFDef}
F(x) \coloneqq \frac{M(x) x}{\mu(x)} \quad \text{with} \quad \mu(x) \coloneqq M(x) \frac{x}{\vert x \vert} \cdot \frac{x}{\vert x \vert},
\end{equation}
where the matrix field $M(x)$ is given by \eqref{eq:elliptic_matrix}, so that
\begin{equation}\label{eqn:InnerVariationFieldFProperties}
F(x) \cdot \frac{x}{\vert x \vert} = \vert x \vert \ \text{ in }\ B_1 \qquad \text{and} \qquad F(x) \cdot e_d = 0\ \text{ on }\ B_1',
\end{equation}
where the last property follows from the fact that $e_d$ is an eigenvector of $M$. Let $\mathcal F$ and $w$ be as in \cref{thm:thinobs}, and $w_{\varepsilon}$ be defined as
\[
w_{\varepsilon} \coloneqq w \circ T_{\varepsilon}^{-1}.
\]

Given any function $L=L(x,y,p)$, with variables $x\in\R^d$, $y\in\R$, $p\in\R^d$, a standard change of coordinates and direct differentiation yield
\begin{align*}
        \begin{split}
    \int_{B_r^+} L\left( x, w_{\varepsilon}, \nabla w_{\varepsilon} \right)
        &= \int_{B_r^+} L\left( x, w \circ T_{\varepsilon}^{-1}, D(T_{\varepsilon}^{-1}) (\nabla w \circ T_{\varepsilon}^{-1}) \right)\\
        &= \int_{B_r^+} L\left( T_{\varepsilon}, w, (DT_{\varepsilon})^{-1} \nabla w \right) |\det(DT_\varepsilon)| \\
            & = \int_{B_r^+} \Big[ L\left( x, w, \nabla w \right) + \varepsilon \nabla_x L\left( x, w, \nabla w \right) \cdot \psi_r F \\
            &\qquad\qquad - \varepsilon \nabla_p L\left( x, w, \nabla w \right) \cdot D(\psi_r F) \nabla w \Big] \left[ 1 + \varepsilon \diverg\left( \psi_r F \right) \right] + o(\varepsilon) \\
            & = \int_{B_r^+} L\left( x, w, \nabla w \right) + \varepsilon \int_{B_r^+} \Big[ \nabla_x L\left( x, w, \nabla w \right) \cdot \psi_r F + L\left( x, w, \nabla w \right) \diverg\left( \psi_r F \right) \\
            &\quad\qquad\qquad\qquad\qquad- \nabla_p L\left( x, w, \nabla w \right) \cdot D(\psi_r F) \nabla w \Big] + o(\varepsilon).
        \end{split}
    \end{align*}
Applying this formula to $L(x,y,p)=p\cdot M(x)\,p +\partial_dQ(x)\,y\,p_d$  gives    
\begin{align*}
\frac{d}{d\varepsilon}\Big|_{\varepsilon=0}\mathcal F(w_{\eps};r)
   &= (d-2) \int_{B_r^+} \varphi\left( \frac{\vert x \vert}{r} \right) M(x) \nabla w \cdot \nabla w + \frac{1}{r} \int_{B_r^+} \varphi'\left( \frac{\vert x \vert}{r} \right) \vert x \vert M(x) \nabla w \cdot \nabla w   \\
    &\qquad  - \frac{2}{r} \int_{B_r^+} \varphi'\left( \frac{\vert x \vert}{r} \right) \frac{\vert x \vert}{\mu(x)} \left(M(x) \nabla w \cdot \frac{x}{\vert x \vert} \right)^2+E_M+E_F\\
            &\qquad  +\int  \nabla \partial_d Q(x)\cdot (\psi_rF)\, w\,\partial_dw+\partial_d Q(x)\, w\, \partial_d w\, {\rm div}(\psi_r F)-w\,\partial_dQ\partial_d(\psi_r F)\cdot\nabla w
\end{align*}
where $E_M$ and $E_F$ are the following error terms with estimates:
    \begin{align}\label{eqn:onePhaseInnerVariationErrE_h^{i,1}Def}
        \begin{split}
            E_M \coloneqq \int_{B_r^+}  \nabla_x M(x)[\nabla w,\nabla w, \psi_r F], 
            \end{split}
    \end{align}
    which, since $|F(x)|\leq |x|$, can be estimated by
    \[
    |E_M|\leq C r\, \int_{B_r^+} \varphi\left( \frac{\vert x \vert}{r} \right) |\nabla w|^2\,;
    \]
for $E_F$, on the other hand, we have:    
\begin{align}\label{eqn:onePhaseInnerVariationErrE_h^{i,2}Def}
        \begin{split}
            E_F \coloneqq \int_{B_r^+} \varphi\left( \frac{\vert x \vert}{r} \right) &  \nabla w\cdot M(x)\nabla w  \diverg\left( F(x) - x \right) \\
            & - 2\int_{B_r^+} \varphi\left( \frac{\vert x \vert}{r} \right) (M(x)\nabla w)  \cdot D(F(x) - x) \nabla w,
            \end{split}
    \end{align}
which, since $\frac{M(x)}{\mu(x)}$ is smooth and $\frac{M(x)}{\mu(x)}=I+O(x)$ by \eqref{e:PhiProperties-Phi0} and since $|\nabla m|=1$ and $\partial_i \phi(0)=0$ for every $i=1,\dots, d-1$, gives 
\[
\partial_i\left(\frac{M(x)x}{\mu(x)}\right)-e_i=\partial_i\left(\frac{M(x)}{\mu(x)}\right)\,x+\left(\frac{M(x)}{\mu(x)}e_i-e_i\right)=\partial_i(M(x)/\mu(x))\,x+O(|x|)=O(|x|) 
\]
 so that
 \[
 |E_F|\leq C\,r\, \int_{B_r^+} \varphi\left( \frac{\vert x \vert}{r} \right) |\nabla w|^2\,. 
 \]
Now, an integration by parts, and the fact that $F\cdot e_d=0$ on $B_1'$ yield
\begin{align*}
    &\int_{B_r^+}  \nabla \partial_d Q(x)\cdot (\psi_rF)\, w\,\partial_dw+\partial_d Q(x)\, w\, \partial_d w\,{\rm div}(\psi_r F)-w\,\partial_dQ\,\partial_d(\psi_r F)\cdot\nabla w\\
    &\quad=\underbrace{\int_{B_r^+} {\rm div}\left(  \partial_d Q(x)\,\psi_rF\, w\,\partial_dw\right)}_{=0}\\
    &\qquad\qquad-\int_{B_r^+} \left(\partial_dQ\,(\psi_r F)\cdot\nabla w\,\partial_dw+\partial_dQ\,w\,\psi_r F\cdot\partial_d\nabla w+\partial_dQ\,\partial_d(\psi_r F)\cdot\,\nabla w\,w\right)\\
    &\quad =-\int_{\pi_0}\varphi\left(\frac{|x|}r\right)\,\partial_dQ\,w\,F\cdot \nabla w+\underbrace{\int_{B_r^+}\varphi\left(\frac{|x|}r\right) \,\partial_d^2Q(x)w\,F\cdot \nabla w}_{E_Q(r)}\,.
\end{align*}
Integrating by parts on $\pi_0$ and using that $F\cdot \nabla w=F\cdot \nabla_{x'} w$, we get
\begin{align*}
\int_{\pi_0}&\varphi\left(\frac{|x|}r\right)\,\partial_dQ\,w\,F\cdot \nabla w= \frac12\int_{\pi_0}\varphi\left(\frac{|x|}r\right)\,\partial_dQ\,F\cdot \nabla_{x'} (w^2)\\
&= -\frac12\int_{\pi_0}\frac1r\varphi'\left(\frac{|x|}r\right)\,\partial_dQ\,F\cdot \frac{x}{|x|}\, w^2-\frac12\int_{\pi_0}\varphi\left(\frac{|x|}r\right)\,{\rm div}_{x'}(\partial_dQ\,F)\, w^2\\
&= -\frac12\int_{\pi_0}\frac1r\varphi'\left(\frac{|x|}r\right)\,\partial_dQ\,|x|\, w^2-\frac12\int_{\pi_0}\varphi\left(\frac{|x|}r\right)\,{\rm div}_{x'}(\partial_dQ\,F)\, w^2.
\end{align*}
Setting $E^{i,1}=E_M+E_F+E_Q$ the estimate on $E^{i,1}$ follows.
Finally, by applying the above procedure to 
$L(x,y,p)=g_k(x,y,p)P_k(p)y^{3-k},$ we get the remaining desired error estimates.

Putting everything together in terms of $D,A,B$ we get: 
\begin{align}
 0&= (d-2) D_i(r) + \frac{1}{r} \int_{B_r^+} \varphi'\left( \frac{\vert x \vert}{r} \right) \vert x \vert M(x) \nabla w \cdot \nabla w -\frac1r \int_{\pi_0} |x|\,\varphi'\left(\frac{|x|}{r}\right) \,\partial_dQ\,w^2 \notag\\
    &\qquad\qquad  +\frac{2}{r} A(r) -\frac12\int_{\pi_0}\varphi\left(\frac{|x|}r\right)\,{\rm div}(F\partial_dQ)\,w^2
            +\sum_{k=1}^3 E^{i, k}(r).\label{eq:innervar-levico}
\end{align}
On the other hand, computing the derivative of $D(r)$ with respect to $r$, gives: 
  \begin{equation}\label{e:derivative-of-D-raw}
    D'(r) = -\frac{1}{r^2} \int_{\R^d_+} \varphi'\left( \frac{\vert x \vert }{r} \right) \vert x \vert M \nabla w \cdot \nabla w + \frac{1}{2 r^2} \int_{\pi_0} \varphi'\left( \frac{\vert x \vert }{r} \right) \vert x \vert \partial_dQ(x) w^2.
    \end{equation}
    Combining \eqref{e:derivative-of-D-raw} with \eqref{eq:innervar-levico} gives:
    \begin{align*}
 rD'(r)&= (d-2) D_i(r) +\frac{2}{r} A(r) -\frac12\int_{\pi_0}\varphi\left(\frac{|x|}r\right)\,{\rm div}(F\partial_dQ)\,w^2
            +\sum_{k=1}^3 E^{i, k}(r).
\end{align*}
Finally, using the definition \eqref{e:definition-of-D-boundary} of $D_b(r)$ and the fact that 
$$D_i(r)=D(r)-D_b(r)=D(r)+\frac12 \int_{\R^{d-1}\times\{0\}}\varphi\left(\frac{\vert x \vert}{r} \right)\,\partial_d Q(x)\,w^2,$$
we get 
\begin{align*}
 rD'(r)&= (d-2) D(r) +\frac{2}{r} A(r)+e_I(r),
 \end{align*}
where 
\begin{align*}
e_I(r):=\frac12\int_{\pi_0} \varphi\left(\frac{|x|}r\right) ((d-2)\partial_dQ-{\rm div}(\partial_dQ\,F))\,w^2
            +\sum_{k=1}^3 E^{i, k}(r).
\end{align*}

\section{Error estimates}\label{section:ErrorEstimates}

In this section we will provide estimates for the errors in \cref{lem:freqidentities}. In particular we will prove the following
\begin{proposition}[Error estimates]\label{prop:errors}
    There exist constants $R=R(d, \delta_0)>0$, $\kappa=\kappa(d, \delta_0)>0$ and $C = C(d, \delta_0)>0$ with the following property. Suppose that $w$ is a solution of \eqref{eqn:thinObstacleRobinRegularity} under \cref{assumptions:wC1aAprioriEst} and \cref{assumptions:startingAssumptions}. Then, for all $r \in (0, R)$ the following estimates hold:
    \begin{align}
    & \left\vert E^{o, 1}(r) \right\vert \le C D_i(r)^{1+\kappa}, \label{eqn:frequencyErrorsE^{o,1}Est}\\
    & \left\vert E^{o, 2}(r) \right\vert \le C \left[H(r) D_i(r)^{2 \kappa} D_i'(r)\right]^{\sfrac12}, \label{eqn:frequencyErrorsE^{o,2}Est} \\
    & \left\vert E^{i, 1}(r) \right\vert \le C r D_i(r), \label{eqn:frequencyErrorsE^{i,1}Est} \\
    & \left\vert E^{i, 2}(r) \right\vert \le C D_i(r)^{1+\kappa}, \label{eqn:frequencyErrorsE^{i,2}Est} \\
    & \left\vert E^{i, 3}(r) \right\vert \le C \, r \left[ D_i(r)^{\kappa} D_i'(r) + D(r)^{1+\kappa} \right]. \label{eqn:frequencyErrorsE^{i,3}Est}
    \end{align}
\end{proposition}

\subsection{Weighted Poincar\'e and trace inequalities}
In this subsection we gather some inequalities that will be used several times in the proof of \cref{prop:errors}.

\subsubsection*{Height inequality}

We start by showing a bound for the weighted $L^2$ norm in terms of the height function.

\begin{lemma}[Height function integral inequality]\label{lemma:height-function-integral-inequality}
 There is a constant $C_\mu$, depending only on the function $\mu$ defined in \eqref{e:definition-of-mu}, such that for every function $w\in H^1(B_r^+)$ it holds
    \begin{gather}
       G(r)= \int_{\R^d_+} \varphi\left(\frac{|x|}{r}\right)\,w^2\le C_\mu\int_0^r\,H(\rho)\,d\rho,
    \end{gather}
    where $H$ is the height function from \eqref{eqn:onePhaseFrequencyHeightDef}.
\end{lemma}
\begin{proof}
Since 
$$\frac{\partial}{\partial r}\left[\varphi\left(\frac{|x|}{r}\right)\right]=-\frac{|x|}{r^2}\varphi'\left(\frac{|x|}{r}\right),$$
it is immediate to check that 
\begin{align*}
\int_0^rH(s)\,ds&=- \int_0^r\int_{\R^d_+} \frac{1}{\vert x \vert}\varphi'\left( \frac{\vert x \vert }{s} \right) \mu(x) w^2(x)\,dx\,ds\\
&\ge - \int_0^r\int_{\R^d_+} \frac{|x|}{s^2}\varphi'\left( \frac{\vert x \vert }{s} \right) \mu(x) w^2(x)\,dx\,ds\\
&\ge \big(\min_{x\in B_r^+}\mu(x)\big)\int_{\R^d_+} \varphi\left( \frac{\vert x \vert }{r} \right) w^2(x)\,dx\,,
\end{align*}
which concludes the proof.
\end{proof}

\subsubsection*{Trace inequality}
\begin{lemma}[A weighted trace inequality]\label{lemma:weighted-trace-inequality}
	Let $r\ge 0$, and let $\psi_r:[0,+\infty)\to\R$ be a smooth, nonnegative and decreasing function such that 
	$$\psi_r\equiv 1\quad\text{on}\quad [0,\frac{r}{2}]\qquad\text{and}\qquad\psi_r\equiv 0\quad\text{on}\quad [r,+\infty).$$ 
 Then, for every function $w\in H^1(B_r^+)$, we have 
	$$\int_{B_r'}\psi_r(|x|)w^2(x',0)\,dx'\le C_d\left(\frac1r\int_{B_r^+}\psi_r(|x|)w^2(x)\,dx+r\int_{B_r^+}\psi_r(|x|)|\nabla w|^2(x)\,dx\,\right),$$
	where $C_d$ is a dimensional constant.
\end{lemma}
\begin{proof}
For every $x'\in B_{r}'$ and $s\in(0, r)$ we have that the point $(x'(1-s/r),s)$ belongs to $B_r^+$. Indeed,
$$|x'|^2(1-s/r)^2+s^2\le r^2(1-s/r)^2+s^2=(r-s)^2+s^2\le r^2.$$
For every $s\in(0,r)$, we have the estimate
	$$v^2(x',0)-v^2(x'(1-s/r),s)\le 4\int_0^{s}|v|(x'(1-t/r),t)|\nabla v|(x'(1-t/r),t)\,dt,$$
which implies
	$$v^2(x',0)\le \frac{8}{r}\int_{0}^{r/8}v^2(x'(1-t/r),t)\,dt+4\int_0^{r/8}|v|(x'(1-t/r),t)|\nabla v|(x'(1-t/r),t)\,dt.$$
	Multiplying both sides by $\psi_r(|x'|)$ we obtain 
	\begin{align*}
		\psi_r(|x'|)v^2(x',0)&\le \frac{8}{r}\int_{0}^{r/8}\psi_r(|x'|)v^2(x'(1-t/r),t)\,dt\\
		&\qquad+4\int_0^{r/8}\psi_r(|x'|)|v|(x'(1-t/r),t)|\nabla v|(x'(1-t/r),t)\,dt.
	\end{align*}
	We next claim that 
	\begin{equation}\label{e:psi-bound-levico}
		\psi_r(|x'|)\le \psi_r(|(x'(1-t/r),t)|)\quad\text{for every}\quad (x',t)\in B_{r}'\times (0, r/8).
	\end{equation}
For $|x'|^2\ge tr$ we have
	\begin{align*}
		|x'|^2(1-t/r)^2+t^2
		&=|x'|^2+t^2\Big(1+\frac{|x'|^2}{r^2}-2\frac{|x'|^2}{rt}\Big)\\
		&\le|x'|^2+t^2\Big(1-\frac{|x'|^2}{rt}\Big)\le |x'|^2,
	\end{align*}
which (since $\psi_r$ is decreasing) implies that in this case 
	$$\psi_r(|x'|)\le \psi_r(|(x'(1-t/r),t)|).$$
	Conversely, when $|x'|^2\le tr$, we have that 
	\begin{align*}
		|(x'(1-t/r),t)|&\le \sqrt{|x'|^2+t^2}\le \sqrt{tr+t^2}\le \sqrt{2tr}\le \frac{r}{2}.
		\end{align*}
		 so in this case $\psi_r(|x'|)=\psi_r(|(x'(1-t/r),t)|)=1$. This concludes the proof of \eqref{e:psi-bound-levico}. As a consequence, we obtain
	\begin{align*}
		\psi_r(|x'|)v^2(x',0)&\le \frac{8}{r}\int_{0}^{r/8}\psi_r(|(x'(1-t/r),t)|)v^2(x'(1-t/r),t)\,dt\\
		&\qquad+4\int_0^{r/8}\psi_r(|(x'(1-t/r),t)|)|v|(x'(1-t/r),t)|\nabla v|(x'(1-t/r),t)\,dt.
	\end{align*}
	We next integrate in $x'$ over $B_{r}'$ and we use the change of variables 
	$$\Phi(x',t)=(x'(1-t/r),t)\quad\text{defined for}\quad(x',t)\in B_{r}'\times (0, r/8).$$
 Since, for every $(x',t)\in B_{r}'\times (0, r/8)$ we have
 $$(\sfrac12)^{d-1}\le (1-t/r)^{d-1}\le \det \big(D\Phi(x',t)\big)\le 1,$$
	we get $$\int_{B_r'}\psi_r(|x|)w^2(x',0)\,dx'\le 2^{d-1}\left(\frac8r\int_{B_r^+}\psi_r(|x|)w^2(x)\,dx+4\int_{B_r^+}\psi_r(|x|)|w(x)||\nabla w|(x)\,dx\right)$$
	which, by Cauchy-Schwarz, gives the claim. 
\end{proof}

\subsubsection*{Poincar\'e inequality}

Now we show a weighted Poincar\'e-type inequality, or frequency lower bound.
\begin{lemma}[Frequency lower bound]\label{lemma:weightedPoincare}
    There exist constants $c=c(d, \delta_0)>0$ and $r_0=r_0(d, \delta_0)>0$ with the following property. Suppose that $w$ is a solution of \eqref{eqn:thinObstacleRobinRegularity} under \cref{assumptions:wC1aAprioriEst} and \cref{assumptions:startingAssumptions}, then
    \begin{equation}\label{eqn:weightedPoincare}
    r\,D_i(r)\geq c\, H(r) \text{ for every } r \in (0, r_0).
    \end{equation}
\end{lemma}
\begin{proof}
To begin with, we observe that there exist $C=C(d, \delta_0)>0$ and $r_0 = r_0(d, \delta_0)>0$ such that
\begin{equation}\label{eqn:freqLowerBoundr}
Cr \int_{B_r^+} \vert \nabla w \vert^2 \ge  \int_{\partial B_r^+} w^2 \quad \text{and} \quad \| \mu \|_{L^{\infty}\left( B_r^+ \right)} \le C \text{ for every } r \in (0, r_0),
\end{equation}
where the function $\mu$ is the one defined in \eqref{e:definition-of-mu}. Suppose for a moment that \eqref{eqn:freqLowerBoundr} holds true, then denoting 
\[
f(r) \coloneqq \int_{B_r^+} \vert \nabla w \vert^2,
\]
from \eqref{eqn:freqLowerBoundr} and \eqref{eqn:onePhaseFrequencyHeightDef} we have
\begin{align*}
H(r)  = - \int_0^r \varphi'\left( \frac{\rho }{r} \right) \frac{1}{\rho} \int_{\partial B_{\rho}^+} \mu w^2 &\le - C^2 \int_0^r \varphi'\left( \frac{\rho }{r} \right) f(\rho) \\
& = C^2 r \int_0^r \varphi\left( \frac{\rho }{r} \right) f'(\rho)\\
& = C^2 r \int_0^r \varphi\left( \frac{\rho }{r} \right) \int_{(\partial B_\rho)^+} \vert \nabla w \vert^2\,d\rho\\
&= C^2 r \int_{B_r^+} \varphi\left( \frac{|x|}{r} \right) \vert \nabla w \vert^2\,dx = C^2 r D_i(r),
\end{align*}
for all $r \in (0, r_0)$ and with $C$ as in \eqref{eqn:freqLowerBoundr}, which gives \eqref{eqn:weightedPoincare}.

In order to conclude the proof we are left to show \eqref{eqn:freqLowerBoundr}. The second condition follows immediately from the definition of the function $\mu$ in \eqref{e:definition-of-mu} and \cref{assumptions:wC1aAprioriEst}. Concerning the first condition we can proceed for instance by contradiction.

Suppose that there exist sequences of radii $r_n \to 0$, non-negative constants $c_n \to 0$ and solutions $w_n$ to \eqref{eqn:thinObstacleRobinRegularity} under \cref{assumptions:wC1aAprioriEst} and \cref{assumptions:startingAssumptions} such that
\begin{equation}\label{eqn:poincareRContrad}
r_n \int_{B_{r_n}^+} \vert \nabla w_n \vert^2 \le c_n \int_{\partial B_{r_n}^+} w_n^2.
\end{equation}
Consider the rescaled sequence $\tilde w_n: B_1^+ \to \R$ defined as
\[
\tilde w_n \coloneqq \frac{1}{r_n} w\left(r_n x \right)
\]
and their normalizations
\[
w'_n \coloneqq \frac{\tilde w_n}{\int_{\partial B_1^+} \left( \tilde w_n \right)^2}.
\]
From the contradiction assumption \eqref{eqn:poincareRContrad} and a standard trace inequality we have that
\[
\| w'_n \|_{H^1\left( B_1^+ \right)} \le C
\]
for dome dimensional constant $C = C(d) > 0$. In particular, from the $C^{1, \alpha}$ regularity in \cref{prop:nonlinearThinObstacleC1aRegularity} we also have that
\[
\| w'_n \|_{C^{1, \alpha}\left( B_{\sfrac12}^+ \right)} \le C
\]
for some constant $C = C(d, \delta_0) > 0$ and some $\alpha \in (0, \sfrac12)$. We deduce that there exist a function $w_{\infty} \in H^1(B_1^+) \cap C^{1, \alpha}\left( B_{\sfrac12}^+ \right)$ with the following properties:
\begin{itemize}
    \item[(i)] up to a subsequence, $w'_n \to w_{\infty}$ weakly in $H^1(B_1^+)$ and in $C^{1, \alpha}\left( B_{\sfrac12}^+ \right)$;

    \item[(ii)] $w_{\infty}(0) = 0$ thanks to \cref{assumptions:startingAssumptions};

    \item[(iii)] by the compact embedding of $H^1(B_1^+)$ into $L^2(\partial B_1^+)$ we have
    \[
    \int_{\partial B_1^+} w_{\infty}^2 = 1.
    \]

    \item[(iv)] by the contradiction assumption
    \[
    \int_{B_1^+} \vert \nabla w_{\infty}\vert^2 = 0.
    \]
\end{itemize}
From (iv) we deduce that the function $w_{\infty}$ in constant on $B_1^+$, while from $(i)$ and $(ii)$ we have that $w_{\infty} \equiv 0$. This, however, is in contradiction with $(iii)$, which proves \eqref{eqn:freqLowerBoundr}.
\end{proof}
Combining the trace inequality in \cref{lemma:weighted-trace-inequality} with the frequency lower bound in \cref{lemma:weightedPoincare} gives the following 
\begin{corollary}[Boundary energy control]\label{corollary:boundaryEnergyControl}
There exist constants $c=c(d, \delta_0)>0$ and $r_0=r_0(d, \delta_0)>0$ with the following property. Suppose that $w$ is a solution of \eqref{eqn:thinObstacleRobinRegularity} under \cref{assumptions:wC1aAprioriEst} and \cref{assumptions:startingAssumptions}, then
\begin{itemize}
    \item[(i)]
    \[
       G(r) \le c\, r^2 \, D_i(r)\quad \text{ for every }\quad r \in (0, r_0),
    \]

    \item[(ii)]
    \[
       0 < (1-cr) D_i(r)\leq D(r)\leq (1+cr) D_i(r)\, \text{ for every } r \in (0, r_0).
    \]
\end{itemize}
\end{corollary}
\begin{proof}
Starting with point (i), from \cref{lemma:height-function-integral-inequality} and \cref{lemma:weightedPoincare}
\[
G(r) \le C \int_0^r\,H(\rho)\,d\rho \le C \int_{0}^r\rho D_i(\rho)\,d\rho \le C r^2 D_i(r)
\]
for some constant $c=c(d, \delta)$, where in the last inequality we have used the monotonicity of $D_i(r)$. Concerning point (ii), by \cref{lemma:height-function-integral-inequality} and point (i) we have that 
\begin{align*}
D_b(r)&\leq \frac{c}{2}\,r\, D_i(r)+\frac{c}{2r}\, G(r) \le c\,r\, D_i(r),
\end{align*}
which, combined with \eqref{eqn:onePhaseEnergyDef}, immediately implies point (ii).
\end{proof}

\subsection{Grid construction}

For each $j\in \N$ we let $\mathcal{C}_j$ denote the collection of closed cubes in $\R^d$ defined by
\[
L:=[a_1-l_j, \, a_1+l_j]\times  \dots\times [a_d-l_j, \, a_d+l_j]\subset B_1 \,,
\]
where $2l(L):=2l_j=3^{1-j}$ is the side-length of the cube, and $a_i \in 3^{1-j} \mathbb Z$ for every $i = 1, \dots, d$. In the following we will denote with $a(L)=(a_1,\dots,a_d)$ the center of the cube $L\in \mathcal C_j$ and by $B_L$ the ball 
$$B_L:=B_{3 l(L)}\left( a(L) \right)\cap B_1^+.$$ 
Next we set $\mathcal C:=\bigcup_{j\in \N}\mathcal C_j$ and if $H$
and $L$ are two cubes in $\mathcal C$ with $H \subset  L$, then we call $L$ an ancestor of $H$ and $H$ a descendant
of $L$. When in addition $l(L) = 3l(H)$, $H$ is a son of $L$ and $L$ the father of $H$. 

We define a Whitney decomposition of $[-1,1]^d\cap\{x_d>0\}$ as follows.

\begin{definition}[Whitney decomposition]\label{def:whitneyDecompDef}
Let $C_0 > 0$ and $\alpha \in (0, \sfrac12)$ be fixed constants.
We define the family of cubes $\mathcal W= \mathcal W^e\cup \mathcal W^h$ 
inductively as unions 
$$\mathcal W^{e}=\bigcup_jW^{e}_j\qquad\text{and}\qquad\mathcal W^{h}=\bigcup_jW^{h}_j,$$
starting from $W_0=\emptyset$. If no ancestor of $L\in \mathcal C_j$ is in $\mathcal W$, then
\begin{enumerate}
    \item $L\in \mathcal W^e_j$ if
\begin{equation}\label{eqn:excessDecayPropertyDefGrid}
\int_{B_L} \vert \nabla w \vert^2 \geq C_0 \,l(L)^{d+2\alpha}\,, 
\end{equation}    
\item $L\in \mathcal W^h_j$ if $L\notin \mathcal W^e_j$ and 
\begin{equation}\label{eqn:heightDecayPropertyDefGrid}
\int_{B_L}  w^2 \geq C_0 \, l(L)^{d+2+2\alpha} \,.
\end{equation}
\end{enumerate}
If none of the above occurs we say that $L\in \mathcal S_j$. Finally we set 
\[
\Gamma:=\bigcap_{j\ge 1}\bigcup_{L\in \mathcal S_j}L\,.
\]
\end{definition}
We have the following 
\begin{lemma}[Properties of the Whitney decomposition]\label{lem:whitney}
For all constants $N_0 \in \N_{>0}$ and $c_0>0$ there exist constants $\delta_1 = \delta_1(d, C_0, N_0, c_0,  \delta_0)>0$ and $C=C(d, C_0, N_0, \delta_0)>0$ with the following properties. Suppose that the function $w$ is a solution of \eqref{eqn:thinObstacleRobinRegularity} under \cref{assumptions:wC1aAprioriEst} and \cref{assumptions:startingAssumptions}. If the constant $\delta$ from \cref{assumptions:wC1aAprioriEst} is such that $\delta \in (0, \delta_1)$, then the following properties hold:
\begin{enumerate}[(a)]
\item $([-1,1]^{d-1}\times [0,1])=\Gamma\cup \bigcup_{L\in \mathcal W}L$ and the cubes in $\mathcal W$ have disjoint interiors;
       
\item if $x_0\in \Gamma$ then
\begin{equation}\label{eq:contactset}
w(x_0)=0\qquad \text{and}\qquad \nabla w(x_0)=0\,;   
\end{equation}

\item $\mathcal W_j=\emptyset$ for every $j\leq N_0$;

\item if $L\in \mathcal W$, then
\begin{equation}\label{eqn:whitneyCubeSideBoundCenter}
   l(L) \le c_0 \vert a(L) \vert;
\end{equation}

\item if $L\in \mathcal W^e$ and $H$ is the father of $L$, then 
\begin{equation}\label{eqn:excessOnePhaseL2SameScale}
\int_{B_H} w^2 \le C \, l(L)^2 \int_{B_L} \vert \nabla w \vert^2 \quad \text{and} \quad \int_{B_H} \vert \nabla w \vert^2 \le C \int_{B_L} \vert \nabla w \vert^2;
\end{equation}

\item if $L\in \mathcal W^h$ and $H$ is the father of $L$, then 
\begin{equation}\label{eqn:heightOnePhaseL2SameScale}
\int_{B_H} w^2 \le C \int_{B_L} w^2 \quad \text{and} \quad \int_{B_H} \vert \nabla w \vert^2 \le \frac{C}{l(L)^2} \int_{B_L} w^2 .
\end{equation}
\end{enumerate}
\end{lemma}
\begin{proof}
The claims (a), (b), (e) and (f) are a direct consequence of the construction from the above \cref{def:whitneyDecompDef}. In order to prove (c), we notice that \cref{assumptions:wC1aAprioriEst} implies that for any cube $L$
$$\int_{B_L}|\nabla w|^2\le \|\nabla w\|_{L^\infty(B_L)}^2|B_L|\le \delta_1^2|B_L|\le C_d\delta_1^2l(L)^d.$$
Thus, by choosing $\delta_1$ small enough, such that 
$$C_d\delta_1^23^{-dN_0}\le C_03^{-(d+2\alpha)N_0},$$ we get that for $l(L)\ge 3^{-N_0}$ the stopping condition \eqref{eqn:excessDecayPropertyDefGrid} is not verified, so $L\notin \mathcal W^e_j$ for $j=1,\dots,N_0$. The proof of the fact that for $\delta_1$ small enough $L\notin \mathcal W^h_j$ for $j=1,\dots,N_0$ is analogous.

Finally, we prove (d). Suppose that $L\in\mathcal W^e$ and that
$l(L)\ge c_0|a(L)|.$
Then the stopping condition implies 
$$\int_{B_L} \vert \nabla w \vert^2 \geq C_0 \,l(L)^{d+2\alpha}.$$
On the other hand, \cref{assumptions:wC1aAprioriEst} and the fact that $B_L$ is contained in the ball of center $0$ and radius $a(L)+3l(L)$ implies that
\begin{align*}
\int_{B_L} \vert \nabla w \vert^2 &\leq |B_L|\delta_1^2\,(a(L)+3l(L))^{2\alpha}\\
&\leq \delta_1^2 \,C_d\,l(L)^{d}\,(a(L)+3l(L))^{2\alpha} \leq \delta_1^2\,C_d\,\left(3+\frac1{c_0}\right)^{2\alpha}l(L)^{d+2\alpha},
\end{align*}
which contradicts the stopping condition when $\delta_1$ is small enough. Analogously, it is  not possible to have $L\in \mathcal W_h$ and $l(L)\ge c_0|a(L)|.$
\end{proof}

\subsection{Estimates: proof of \cref{prop:errors}}
In order to prove \cref{prop:errors}, we will use repeatedly the next lemma, which deals with the uniform estimates on the cubes. This will allow us to treat the error terms arising from the nonlinearities and is the key for the frequency monotonicity. Its proof is a simple consequence of the grid construction, which guarantees that the doubling assumptions of \cref{lemma:excessOnePhaseThreeAnnuliLemma} are satisfied, and the regularity of minimizers.
\begin{lemma}[$L^{\infty}-L^2$ estimates on the cubes]\label{lemma:energyEstCubesHD}
There exist constants $R=R(d, \delta_0)>0$, $\kappa=\kappa(d, \delta_0)>0$ and $C = C(d, \delta_0)>0$ with the following property. Suppose that $w$ is a solution of \eqref{eqn:thinObstacleRobinRegularity} under \cref{assumptions:wC1aAprioriEst} and \cref{assumptions:startingAssumptions}. Then, for all cubes $L \in \mathcal{W}$ with
\[
L \cap B_r^+ \neq \emptyset,
\]
the following estimates hold:
\begin{itemize}
    \item[(i)] 
    \[
    \| w \|_{L^{\infty}\left( L \right)}\le C D_i(r)^{\kappa} \quad \text{for all } r \in (0, R),
    \]

    \item[(ii)]
    \[
    \| \nabla w \|_{L^{\infty}\left( L \right)} \le C D_i(r)^{\kappa} \quad \text{for all } r \in (0, R).
    \]
\end{itemize}
\end{lemma}
\begin{proof}
To begin with, let us choose $C_0 = 1$ in \cref{def:whitneyDecompDef}, $N_0 = 1$ and $c_0 = 1/10$ in \cref{lem:whitney} and $\delta>0$ in \cref{assumptions:wC1aAprioriEst} small enough (depending now only on $d$ and $\delta_0$) so that \eqref{eqn:whitneyCubeSideBoundCenter} holds true.

We distinguish two cases:
\begin{enumerate}
\item Suppose first that the cube $L$ is of type excess, i.e. that $L \in \mathcal W^e$. From \cref{prop:nonlinearThinObstacleC1aRegularity}, we have the estimates
\begin{equation}\label{eqn:excessOnePhaseL2LinftySameScalew}
\| w \|_{L^{\infty}\left(L \right)} \le C l(L) \left( \frac{1}{l(L)^d} \int_{L} \vert \nabla w \vert^2  \right)^{\sfrac12} \le C l(L)^{1+\alpha}
\end{equation}
and
\begin{equation}\label{eqn:excessOnePhaseL2LinftySameScaleNablaw}
\| \nabla w \|_{L^{\infty}\left(L \right)} \le C \left( \frac{1}{l(L)^d} \int_{L} \vert \nabla w \vert^2  \right)^{\sfrac12} \le C l(L)^{\alpha},
\end{equation}
for some constant $C=C(d, \delta_0)>0$. Thanks to \eqref{eqn:excessOnePhaseL2SameScale}, the hypotheses of \cref{lemma:excessOnePhaseThreeAnnuliLemma} are satisfied, with the corresponding constant $C=C(d, \delta_0)>0$. Hence, let us choose $\eta = 1/10$ in \cref{lemma:excessOnePhaseThreeAnnuliLemma} and let $R(d, \delta_0) = r_0(d, \eta, C, \delta_0)$, where $r_0$ is the one from \cref{lemma:excessOnePhaseThreeAnnuliLemma}. Moreover, we choose the center $x'$ so that
\begin{equation}\label{eqn:excessOnePhaseBoundaryCotrolBall}
\dist\left(B_{3\eta l(L)}(x'), \, \partial B_r^+ \right) \ge l(L)/2.
\end{equation}
Notice we need to choose $x' \neq a(L)$ only when $B_L \cap \partial B \cap \{ x_d \ge 0 \} \neq \emptyset$.

Using \cref{lemma:excessOnePhaseThreeAnnuliLemma} with the above choices, we can estimate
\begin{equation}\label{eqn:excessOnePhaseDoublingCube}
    \int_{B_{3\eta l(L)}(x')} \vert \nabla w \vert^2 \ge \gamma \int_{B_L} \vert \nabla w \vert^2 \ge \gamma l(L)^{d+2\alpha},
\end{equation}
where $\gamma = \gamma(d, \delta_0)>0$. Moreover, thanks to \eqref{eqn:excessOnePhaseBoundaryCotrolBall}, from \eqref{eqn:onePhaseCutoffDef} we see that (independently of $\upsilon$)
\begin{equation}\label{eqn:excessOnePhaseBoundaryCotrolCutoff}
\varphi\left( \frac{\vert x \vert}{r} \right) \ge l(L)/2.
\end{equation}
Combining \eqref{eqn:excessOnePhaseDoublingCube} with \eqref{eqn:excessOnePhaseBoundaryCotrolCutoff} and the definition \eqref{e:definition-of-D-interior} gives
\begin{equation}\label{eqn:excessOnePhaseDoublingCube2}
    D_i(r) \ge \int_{B_{3\eta l(L)}(x')} \varphi\left( \frac{\vert x \vert}{r} \right) \vert \nabla w \vert^2 \ge \frac{\gamma}{2} l(L)^{d+2\alpha+1},
\end{equation}
Plugging \eqref{eqn:excessOnePhaseDoublingCube2} into \eqref{eqn:excessOnePhaseL2LinftySameScalew} and \eqref{eqn:excessOnePhaseL2LinftySameScaleNablaw} we obtain
\[
\| w \|_{L^{\infty}\left(L \right)} \le C \left( \frac{2}{\gamma} \right)^{\kappa} D_i^{\kappa} \quad \text{with } \kappa = \frac{1+\alpha}{d+2\alpha+1},
\]
and
\[
\| \nabla w \|_{L^{\infty}\left(L \right)} \le C \left( \frac{2}{\gamma} \right)^{\kappa} D_i^{\kappa} \quad \text{with } \kappa = \frac{\alpha}{d+2\alpha+1},
\]
which concludes the proof of cases (i) and (ii) if $L \in \mathcal W^e$ is a cube of type excess.

\item Suppose now that $L$ is a cube of type excess, i.e. that $L \in \mathcal W^h$. We can proceed similarly as in the previous case. From \cref{prop:nonlinearThinObstacleC1aRegularity}, we have analogous estimates estimates
\begin{equation}\label{eqn:heightOnePhaseL2LinftySameScalew}
\| w \|_{L^{\infty}\left(L \right)} \le C \left( \frac{1}{l(L)^d} \int_{L} w^2  \right)^{\sfrac12} \le C l(L)^{1+\alpha}
\end{equation}
and
\begin{equation}\label{eqn:heightOnePhaseL2LinftySameScaleNablaw}
\| \nabla w \|_{L^{\infty}\left(L \right)} \le \frac{C}{l(L)} \left( \frac{1}{l(L)^d} \int_{L} w^2  \right)^{\sfrac12} \le C l(L)^{\alpha},
\end{equation}
for some constant $C=C(d, \delta_0)>0$. Using this time \cref{lemma:heightOnePhaseThreeAnnuliLemma}, with the same choices of parameters as in case 1, we deduce the inequality
\begin{equation}\label{eqn:heightOnePhaseDoublingCube2}
    G(r) \ge \int_{B_{3\eta l(L)}(x')} \varphi\left( \frac{\vert x \vert}{r} \right) w^2 \ge \frac{l(L)}{2} \gamma \int_{B_L} w^2 \ge \frac{\gamma}{2} l(L)^{d+2\alpha+2},
\end{equation}
for some constant $\gamma = \gamma(d, \delta_0)$. Moreover, \cref{lemma:height-function-integral-inequality} and \cref{lemma:weightedPoincare} imply that
\begin{equation}\label{eqn:onePhaseFrequencyL2DboundCubes}
    G(r) \le C r^2 D_i(r) \quad \text{for all } r \in (0, r_0),
\end{equation}
where $C=C(d, \delta_0)>0$ and $r_0 = r_0(d, \delta_0)>0$. In particular, the choice of $R=R(d, \delta_0)>0$ we can meke in this case is the minimum between the radii $r_0$ from \cref{lemma:weightedPoincare} and \cref{lemma:heightOnePhaseThreeAnnuliLemma}. Combining \eqref{eqn:onePhaseFrequencyL2DboundCubes} and \eqref{eqn:heightOnePhaseDoublingCube2} we get
\begin{equation}\label{eqn:heightOnePhaseDoublingCube3}
    C r^2 D_i(r) \ge \frac{\gamma}{2} l(L)^{d+2\alpha+2},
\end{equation}
Now, using \eqref{eqn:heightOnePhaseDoublingCube3} together with \eqref{eqn:heightOnePhaseL2LinftySameScalew} and \eqref{eqn:heightOnePhaseL2LinftySameScaleNablaw}, we obtain the estimates
\[
\| w \|_{L^{\infty}\left(L \right)} \le C^{1+\kappa} \left( \frac{2}{\gamma} \right)^{\kappa} r^{2 \kappa} D_i^{\kappa} \quad \text{with } \kappa = \frac{1+\alpha}{d+2\alpha+2},
\]
and
\[
\| \nabla w \|_{L^{\infty}\left(L \right)} \le C^{1+\kappa} \left( \frac{2}{\gamma} \right)^{\kappa} r^{2 \kappa} D_i^{\kappa} \quad \text{with } \kappa = \frac{\alpha}{d+2\alpha+2},
\]
which concludes the proof of cases (i) and (ii) if $L \in \mathcal W^h$.
\end{enumerate}
\end{proof}

\subsubsection{Outer errors: estimates of $E^{o, k}(r)$}
The aim of this subsection is to estimate the outer error terms $E^{o, k}$ in \cref{prop:errors}. We begin with the estimate \eqref{eqn:frequencyErrorsE^{o,1}Est}.
\begin{lemma}[Estimate of $E^{o,1}$]\label{lemma:frequencyErrorsE^{o,1}Est}
There exist constants $R=R(d, \delta_0)>0$, $\kappa=\kappa(d, \delta_0)>0$ and $C = C(d, \delta_0)>0$ with the following property. Suppose that $w$ is a solution of \eqref{eqn:thinObstacleRobinRegularity} under \cref{assumptions:wC1aAprioriEst} and \cref{assumptions:startingAssumptions}. Then,
\begin{align*}
    \left\vert E^{o, 1}(r) \right\vert \le C D_i(r)^{1+\kappa}\quad \text{for all } r \in (0, R).
\end{align*}
\end{lemma}
\begin{proof}
We recall that by \eqref{eq:bulkouter1} we have  
\begin{align*}   
& |E^{o, 1}(r)|\leq C  \int_{B_r^+} \varphi\left( \frac{\vert x \vert}{r} \right) \left(|w|\, \vert \nabla w \vert^2 + \vert \nabla w \vert^3\right)=C  \int_{B_r^+} \psi_r(x) \left(|w|\, \vert \nabla w \vert^2 + \vert \nabla w \vert^3\right),
\end{align*}
so by the Whitney decomposition and \eqref{eq:contactset} we can estimate
\begin{align*}
    \begin{split}
        |E^{o, 1}(r)| &\le C \sum_{L \in \mathcal{W}} \left[  \int_{L} \psi_r w \vert \nabla w \vert^2  +  \int_{L} \psi_r \vert \nabla w \vert^3 \right] \\
        & \le C \sum_{L \in \mathcal{W}} \left[ \| w \|_{L^{\infty}(L)}  \int_{L} \psi_r \vert \nabla w \vert^2  + \| \nabla w \|_{L^{\infty}(L)}  \int_{L} \psi_r \vert \nabla w \vert^2  \right] \\
        & \le C D(r)^{\kappa} \sum_{L \in \mathcal{W}}  \int_{L} \psi_r \vert \nabla w \vert^2  \le C D(r)^{1+\kappa} ,
    \end{split}
\end{align*}
where in the last but one passage we have used \cref{lemma:energyEstCubesHD}. The estimate now follows from \cref{corollary:boundaryEnergyControl}.
\end{proof}
Now we turn to the error $E^{o,2}$, i.e. the estimate \eqref{eqn:frequencyErrorsE^{o,2}Est}.
\begin{lemma}[Estimate of $E^{o,2}$]\label{lemma:frequencyErrorsE^{o,2}Est}
There exist constants $R=R(d, \delta_0)>0$, $\kappa=\kappa(d, \delta_0)>0$ and $C = C(d, \delta_0)>0$ with the following property. Suppose that $w$ is a solution of \eqref{eqn:thinObstacleRobinRegularity} under \cref{assumptions:wC1aAprioriEst} and \cref{assumptions:startingAssumptions}. Then,
\[
\big\vert E^{o, 2}(r) \big\vert \le C \left[H(r) D_i(r)^{2 \kappa} D_i'(r)\right]^{\sfrac12} \quad \text{for all } r \in (0, R).
\]
\end{lemma}
\begin{proof}
By \eqref{eq:bulkouter2} we have that
\begin{align*}   
& |E^{o, 2}(r)| \leq \frac{C}{r} \int_{ B_r^+ } \left|\varphi'\right|\left( \frac{\vert x \vert}{r} \right)\,\left(|w|^2\,|\nabla w|+|w|\,|\nabla w|^2 \right).
\end{align*}
Thanks to the Whitney decomposition and the property \eqref{eq:contactset} we have
\begin{align*}
    \begin{split}
        |E^{o, 2}(r)| &\le \frac{C}{r} \sum_{L \in \mathcal{W}} \left[  \int_{L}  \left|\varphi'\right|\left( \frac{\vert x \vert}{r} \right)\, w \vert \nabla w \vert^2  +  \int_{L}  \left|\varphi'\right|\left( \frac{\vert x \vert}{r} \right)\, w^2 \vert \nabla w \vert  \right] \\
        &\le \frac{C}{r} \sum_{L \in \mathcal{W}} \left[  \| \nabla w \|_{L^{\infty}(L)} +  \| w \|_{L^{\infty}(L)}\right] \int_{L}  \left|\varphi'\right|\left( \frac{\vert x \vert}{r} \right)\, w \vert \nabla w \vert   \\
        & \le \frac{C}{r} D(r)^{\kappa} \int_{B_r^+}  \left|\varphi'\right|\left( \frac{\vert x \vert}{r} \right)\, w \vert \nabla w \vert \\
 & \le \frac{C}{r} D(r)^{\kappa} \left(\int_{B_r^+}  \left|\varphi'\right|\left( \frac{\vert x \vert}{r} \right)\,\frac{w^2}{|x|} \right)^{\sfrac12}\left(\int_{B_r^+}  \left|\varphi'\right|\left( \frac{\vert x \vert}{r} \right)\,{|x|} \vert \nabla w \vert^2\right)^{\sfrac12} \\
 & \le \frac{C}{r} D(r)^{\kappa} \left[H(r)r^2D_i'(r)\right]^{\sfrac12} = C D(r)^{\kappa} \left[ H(r) D_i'(r)\right]^{\sfrac12},
\end{split}
\end{align*}
which, using \cref{corollary:boundaryEnergyControl}, concludes the proof.
\end{proof}

\subsubsection{Inner errors: estimates of $E^{i, k}(r)$}
Now we estimate the inner error terms $E^{i, k}$ in \cref{prop:errors}. We proceed with \eqref{eqn:frequencyErrorsE^{i,1}Est} and \eqref{eqn:frequencyErrorsE^{i,2}Est}.
\begin{lemma}[Estimate of $E^{i,1}$ and $E^{i,2}$]\label{lemma:frequencyErrorsE^{i,1}E^{i,2}Est}
There exist constants $R=R(d, \delta_0)>0$, $\kappa=\kappa(d, \delta_0)>0$ and $C = C(d, \delta_0)>0$ with the following property. Suppose that $w$ is a solution of \eqref{eqn:thinObstacleRobinRegularity} under \cref{assumptions:wC1aAprioriEst} and \cref{assumptions:startingAssumptions}. Then,
\begin{align*}
    &  \big\vert E^{i, 1}(r) \big\vert \le C r D_i(r) \quad \text{for all } r \in (0, R),\\
    & \big\vert E^{i, 2}(r) \big\vert \le C D_i(r)^{1+\kappa} \quad \text{for all } r \in (0, R).
\end{align*}
\end{lemma}
\begin{proof}
Concerning the estimate on the error $E^{i, 1}(r)$, from \eqref{eq:onePhaseE{i,1}} and a Cauchy-Schwarz inequality we have
\begin{align*}
    E^{i, 1}(r) \le C r D(r) + C r \left( G(r) D(r) \right)^{\sfrac12} \le C r D(r) + C r D(r),
\end{align*}
where in the last passage we have used \cref{corollary:boundaryEnergyControl} (i). The estimate now follows using \cref{corollary:boundaryEnergyControl} (ii).

Let us turn to the error $E^{i, 2}(r)$. From \eqref{eq:onePhaseE{i,2}} and using the Whitney decomposition in cubes, taking also into account \eqref{eq:contactset} we can estimate
\begin{align*}
    \begin{split}
        |E^{i, 2}(r)| &\le C \sum_{L \in \mathcal{W}} \left[  \int_{L} \psi_r \vert \nabla w \vert^3  +  \int_{L}  \psi_r w \vert \nabla w \vert^2  +  \int_{L}  \psi_r w^2 \vert \nabla w \vert  \right] \\
        & \le C \sum_{L \in \mathcal{W}} \left[ \| \nabla w \|_{L^{\infty}(L)}  \int_{L}  \psi_r \vert \nabla w \vert^2  + \| w \|_{L^{\infty}(L)}  \int_{L}  \psi_r \vert \nabla w \vert^2  + \| \nabla w \|_{L^{\infty}(L)}  \int_{L}  \psi_r w^2   \right] \\
        & \le C D(r)^{\kappa} \sum_{L \in \mathcal{W}} \left[  \int_{L}  \psi_r \vert \nabla w \vert^2  +  \int_{L}  \psi_r w^2   \right] \\
        & \le C D(r)^k \left[ D(r) + G(r) \right],
    \end{split}
\end{align*}
where in the last but one passage we have used \cref{lemma:energyEstCubesHD}. The estimate for the error $E^{i, 2}(r)$ now follows from \cref{corollary:boundaryEnergyControl}.
\end{proof}
To conclude, we prove the estimate \eqref{eqn:frequencyErrorsE^{i,3}Est}, concerning the inner error $E^{i,3}$.
\begin{lemma}[Estimate of $E^{i,3}$]\label{lemma:frequencyErrorsE^{i,3}Est}
There exist constants $R=R(d, \delta_0)>0$, $\kappa=\kappa(d, \delta_0)>0$ and $C = C(d, \delta_0)>0$ with the following property. Suppose that $w$ is a solution of \eqref{eqn:thinObstacleRobinRegularity} under \cref{assumptions:wC1aAprioriEst} and \cref{assumptions:startingAssumptions}. Then,
\[
\left\vert E^{i, 3}(r) \right\vert \le C \, r \left[ D_i(r)^{\kappa} D_i'(r) + D(r)^{1+\kappa} \right] \quad \text{for all } r \in (0, R).
\]
\end{lemma}
\begin{proof}
From \eqref{eq:onePhaseE{i,3}}, we have 
\begin{align*}
|E^{i, 3}(r)| & \le  \frac{C}{r} \int_{B_r^+} |\varphi'|\left( \frac{\vert x \vert}{r} \right) \vert x \vert \left(|\nabla w|^3+|w|\,|\nabla w|^2+|w|^2\,|\nabla w|\right)\,. 
\end{align*}
By the Whitney decomposition, \eqref{eq:contactset} and \cref{lemma:energyEstCubesHD} we can estimate
\begin{align*}
    \begin{split}
        |E^{i, 3}(r)| &\le \frac{C}{r} \sum_{L \in \mathcal{W}} \left[  \int_{L} |\varphi'|\left( \frac{\vert x \vert}{r} \right) \vert x \vert \vert \nabla w \vert^3  +  \int_{L}  |\varphi'|\left( \frac{\vert x \vert}{r} \right) \vert x \vert w \vert \nabla w \vert^2  +  \int_{L}  |\varphi'|\left( \frac{\vert x \vert}{r} \right) \vert x \vert w^2 \vert \nabla w \vert  \right] \\
        & \le \frac{C}{r} \sum_{L \in \mathcal{W}} \left[ \Big(\| \nabla w \|_{L^{\infty}(L)}   + \| w \|_{L^{\infty}(L)}\Big)  \int_{L}  |\varphi'|\left( \frac{\vert x \vert}{r} \right) \vert x \vert \vert \nabla w \vert^2  + \| \nabla w \|_{L^{\infty}(L)}  \int_{L}  |\varphi'|\left( \frac{\vert x \vert}{r} \right) \vert x \vert w^2   \right] \\
        & \le \frac{C}{r} D(r)^{\kappa} \sum_{L \in \mathcal{W}} \left[  \int_{L}  |\varphi'|\left( \frac{\vert x \vert}{r} \right) \vert x \vert \vert \nabla w \vert^2  +  \int_{L}  |\varphi'|\left( \frac{\vert x \vert}{r} \right) \vert x \vert w^2   \right] \\
        & \le \frac{C}{r} D(r)^k \left[ r^2 D_i'(r) + r^2 H(r) \right] \le C r D(r)^{\kappa} \left[ D_i'(r) + r D(r) \right],
    \end{split}
\end{align*}
where in the last inequality we have used \cref{lemma:weightedPoincare}. Now the estimate follows from \cref{corollary:boundaryEnergyControl}.
\end{proof}

\section{Frequency monotonicity and proof of \cref{thm:onePhaseAnalyticObstacle}}\label{section:FrequencyMonotonicity}
In this section we first show the (almost-)monotonicity of the Almgren-type frequency function \eqref{eqn:onePhaseFrequencyN(r)Def}. Then, we combine such result with a blow-up argument to conclude the proof of \cref{thm:onePhaseAnalyticObstacle}.

\subsection{Frequency (almost-)monotonicity}
We begin with the (almost-)monotonicity of the frequency function.
\begin{theorem}[Frequency monotonicity]\label{thm:freqmon}
   There exist constants $R=R(d, \delta_0)>0$, $\kappa=\kappa(d, \delta_0)>0$ and $C = C(d, \delta_0)>0$ with the following property. Suppose that $w$ is a solution of \eqref{eqn:thinObstacleRobinRegularity} under \cref{assumptions:wC1aAprioriEst} and \cref{assumptions:startingAssumptions}. Then,
   \begin{equation}\label{eqn:onePhaseFrequencyAlmostMonocotone}
        e^{\,g(r)} N(r) \text{ is non-decreasing for all } r \in (0, R),
    \end{equation}
    where the frequency function $N(r)$ is as in \eqref{eqn:onePhaseFrequencyN(r)Def} and the function $g(r): \R^+ \to \R$ is defined as
    \begin{equation}\label{eqn:onePhaseFrequencyg(r)Def}
    g(r) \coloneqq \frac{C}{\kappa} \left[ r^{\kappa} + D(r)^{\kappa} \right] - C \frac{1}{D_i(r)} \int_0^r H(\rho) \, d\rho + C \int_0^r \frac{H(\rho)}{D(\rho)} \, d\rho
    \end{equation}
    and satisfies
    \begin{equation}\label{eqn:onePhaseFrequencyg(r)Limit0}
    g(r) \to 0 \text{ as } r \to 0^+.
    \end{equation}
\end{theorem}
\begin{remark}\label{remark:freqMonotoneIndependent}
    Notice that neither constants $R$, $\kappa$ nor $C$ in \cref{thm:freqmon} depend on the parameter $\upsilon$ in \eqref{eqn:onePhaseCutoffDef}. In particular, the property \eqref{eqn:onePhaseFrequencyAlmostMonocotone} holds true independently of $\upsilon \in (\sfrac12, 1)$.
\end{remark}
\begin{proof}[Proof of \cref{thm:freqmon}]
To begin with, from \cref{remark:wNonDegeneracy} we have that
\begin{equation}\label{eqn:gradientwNotVanishr}
\int_{B_r^+} w^2 > 0 \quad \text{and} \quad \int_{B_r^+} \vert \nabla w \vert^2 > 0 \quad \text{for all } r \in (0, 1),
\end{equation}
while from \cref{lem:freqidentities} we have
\[
H(r) > 0 \quad \text{for all } r \in (0, 1).
\]
Moreover, combining \eqref{eqn:gradientwNotVanishr} with \cref{corollary:boundaryEnergyControl} we have that there exist a constant $R=R(d, \delta_0)$ small enough so that
\[
D(r)>0 \text{ for all } r \in (0, R).
\]
Then, we compute
\begin{align*}
 \frac{d}{dr} \ln{N(r)} &= \frac{1}{r} + \frac{D'(r)}{D(r)} - \frac{H'(r)}{H(r)}\\
 &= \frac{1}{r} + \frac{\frac1r(d-2) D(r)+\frac{2}{r^2} A(r)+\frac1r e_I(r)}{D(r)} - \frac{\frac{d-1}{r} H(r) + \frac{2}{r} B(r) + e_{H}(r)}{H(r)}\\
  &= \frac{2A}{r^2D}+\frac1r\frac{e_I}{D} -\frac{2B}{rH}-\frac{e_H}{H}\\
    &= \frac{2}{r^2}\left(\frac{A}{D}-\frac{rB}{H}\right)+\frac1r\frac{e_I}{D}-\frac{e_H}{H}
\end{align*}
Let us denote with $F(r)$ the quantity
\[
F(r) \coloneqq \frac{1}{r} B(r) - \frac{1}{2} E^{o, 2}(r).
\]
Thanks to \eqref{eqn:onePhaseOuterVariationD}, \eqref{eqn:frequencyErrorsE^{o,1}Est} and \cref{corollary:boundaryEnergyControl} (i) there exist constants $\kappa = \kappa(d, \delta_0)$, $C_0 = C_0(d, \delta_0)$ and $R=R(d, \delta_0)$ such that
\[
\left\vert F(r) - D(r) \right\vert= |G(r)+E^{o,1}(r)| \le C \left[r^2 D(r) + D(r)^{1+\kappa} \right] \quad \text{for all } r \in (0, R),
\]
which implies
\begin{equation}\label{eqn:frequenciMonotoneF(r)D(r)Equiv}
    (1 - c r) D(r) \le F(r) \le (1 + c r) D(r) \quad \text{for all } r \in (0, R)
\end{equation}
and some constant $c=c(d, \delta_0)$.

Then, we can write
\begin{align*}    
& \frac{2}{r^2}\left(\frac{A}{D}-\frac{rB}{H}\right)+\frac1r\frac{e_I}{D}-\frac{e_H}{H} = \\
&= \frac{2}{r^2}\left(\frac{A}{F}-\frac{rB}{H}\right)-\frac{2}{r^2}\left(\frac{A}{F}-\frac{A}{D}\right)+\frac1r\frac{e_I}{D}-\frac{e_H}{H} = \\
&= \frac{2}{r^2 FH}\left(AH-rBF\right)-\frac{2A}{r^2}\left(\frac{D-F}{FD}\right)+\frac1r\frac{e_I}{D}-\frac{e_H}{H}\\
&= \frac{2}{r^2 FH}\left(AH-B^2\right) + \frac{2B}{r^2 FH}\left(B-rF\right)-\frac{2A}{r^2}\left(\frac{D-F}{FD}\right)+\frac{e_I}{rD}-\frac{e_H}{H}\\
&= \frac{2}{r^2 FH}\left(AH-B^2\right)+ \frac{B E^{o,2}}{r FH}-\frac{2A}{r^2FD}\left(G + E^{o,1}\right)+\frac{e_I}{rD}-\frac{e_H}{H}.
\end{align*}
Finally, since by the Cauchy-Schwarz inequality
$AH\ge B^2$, we get that
\begin{align}\label{eqn:frequencyMonotoneOnePhaseEst1}
\begin{split}
 \frac{d}{dr} \ln{N(r)}
 &\ge \frac{B(r) E^{o,2}(r)}{r F(r)H(r)}-\frac{2A(r)}{r^2F(r)D(r)}\left(G(r) + E^{o,1}(r)\right)+\frac{e_I(r)}{wrD(r)}-\frac{e_H(r)}{H(r)}.
\end{split}
\end{align}
Now we estimate separately all the terms at the right hand side:
\begin{itemize}
    \item to begin with, we show that
    \begin{equation}\label{eqn:frequencyMonotoneOnePhaseErrorE^{o,2}}
        \left\vert \frac{B(r) E^{o,2}(r)}{r F(r)H(r)} \right\vert \le C D_i(r)^{\kappa-1} D_i'(r) \quad \text{for all } r\in(0, R)
    \end{equation}
    and some constants $C=C(d, \delta_0)>0$, $R=R(d, \delta_0)>0$, $\kappa=\kappa(d, \delta_0)$. Indeed, from \eqref{e:definition-of-B-levico}, \eqref{eqn:onePhaseFrequencyHeightDef}, \eqref{e:definition-of-D-interior} and a Cauchy-Schwarz inequality we have
    \begin{equation}\label{eqn:frequenciMonotoneB(r)Est}
    \vert B(r) \vert \le C \left( r^2 H(r) D_i'(r) \right)^{\sfrac12} \quad \text{for all } r\in(0, R)
    \end{equation}
    and some constants $C=C(d, \delta_0)>0$, $R=R(d, \delta_0)>0$, $\kappa=\kappa(d, \delta_0)$. Combining the estimates \eqref{eqn:frequenciMonotoneB(r)Est} and \eqref{eqn:frequenciMonotoneF(r)D(r)Equiv} with the error estimate  \eqref{eqn:frequencyErrorsE^{o,2}Est} from \cref{prop:errors} gives \eqref{eqn:frequencyMonotoneOnePhaseErrorE^{o,2}};

    \item now we prove that 
    \begin{equation}\label{eqn:frequencyMonotoneOnePhaseErrorsE^{o,1}G}
        \left\vert \frac{2A(r)}{r^2F(r)D(r)}\left(G(r) + E^{o,1}(r)\right) \right\vert \le C\, D_i'(r)\left(\frac{1}{D_i(r)^2}\int_0^r H(\rho) \, d\rho  +  D_i(r)^{\kappa-1}\right) \quad \text{for all } r\in(0, R)
    \end{equation}
    and some constants $C=C(d, \delta_0)>0$, $R=R(d, \delta_0)>0$, $\kappa=\kappa(d, \delta_0)$. 
    
    By \eqref{e:definition-of-A-levico} we have that
    \begin{equation}\label{eqn:onePhaseFrequencyA(r)D'(r)Est}
    \vert A(r) \vert \le C\, D_i'(r)\left[\frac{1}{D_i(r)^2}\int_0^r H(\rho) \, d\rho  +  D_i(r)^{\kappa-1}\right] \quad \text{for all } r\in(0, R)
    \end{equation}
    Combining \eqref{eqn:onePhaseFrequencyA(r)D'(r)Est} with the error estimate \eqref{eqn:frequencyErrorsE^{o,1}Est} form \cref{prop:errors} and the estimate on $G(r)$ from \cref{lemma:height-function-integral-inequality}, we get 
    \begin{align*}
    \left\vert \frac{2A(r)}{r^2F(r)D(r)}\left(G(r) + E^{o,1}(r)\right) \right\vert &\le\frac{2|A(r)|}{r^2F(r)D(r)}\left(C \int_0^r H(\rho) \, d\rho + C D_i(r)^{1+\kappa}\right)\\
    &\le C\frac{ D_i'(r)}{F(r)D(r)}\left( \int_0^r H(\rho) \, d\rho +  D_i(r)^{1+\kappa}\right)\\
        &\le C\, D_i'(r)\left(\frac{1}{D_i(r)^2}\int_0^r H(\rho) \, d\rho  +  D_i(r)^{\kappa-1}\right)
    \end{align*}
    and some constants $C=C(d, \delta_0)>0$, $R=R(d, \delta_0)>0$, $\kappa=\kappa(d, \delta_0)$, where in the last inequality we have used \eqref{eqn:frequenciMonotoneF(r)D(r)Equiv} and \cref{corollary:boundaryEnergyControl} (ii). 

    \item let us show that
    \begin{equation}\label{eqn:frequencyMonotoneOnePhaseErrorse_I}
        \left\vert \frac{e_I}{rD} \right\vert \le C \left[ r^{\kappa - 1} + D_i(r)^{\kappa-1} D_i'(r) \right] \quad \text{for all } r\in(0, R)
    \end{equation}
    and some constants $C=C(d, \delta_0)>0$, $R=R(d, \delta_0)>0$, $\kappa=\kappa(d, \delta_0)$.
    
    First, combining the trace inequality from  \cref{lemma:weighted-trace-inequality} with the Poincaré inequality from \cref{corollary:boundaryEnergyControl} (i) we have that
    \[
    \int_{B_r'} \varphi\left(\frac{|x|}r\right)\,w^2 \le C \left[ r D_i(r) + \frac{1}{r} G(r) \right] \le C r D_i(r),  \quad \text{for all } r\in(0, R)
    \]
    and some constants $C=C(d, \delta_0)>0$, $R=R(d, \delta_0)>0$. Combining the estimate \eqref{eq:innervar-error-levico} for $e_I$ with the above inequality and the error estimates \eqref{eqn:frequencyErrorsE^{i,1}Est}, \eqref{eqn:frequencyErrorsE^{i,2}Est}, \eqref{eqn:frequencyErrorsE^{i,3}Est}
    , we get 
    \begin{align*}
    |e_I(r)|&\le C\, \int_{B_r'} \varphi\left(\frac{|x|}r\right)\,w^2
            +\sum_{k=1}^3 |E^{i, k}(r)|\\
            &\le C r D_i(r)+C D_i(r)^{1+\kappa}+C \, r \left[ D_i(r)^{\kappa} D_i'(r) + D(r)^{1+\kappa} \right].
    \end{align*}
  Since $D_i(r)$ is controlled by a power of $r$ thanks to \cref{assumptions:wC1aAprioriEst} and \cref{assumptions:startingAssumptions}, up to changing $\kappa$ and $C$, we get
  \begin{align*}
  |e_I(r)| & \le C r^{\kappa} D_i(r) + C \, r D_i(r)^{\kappa} D_i'(r),  \quad \text{for all } r\in(0, R),
  \end{align*}
  with constants $R=R(d, \delta_0)>0$, $C=C(d, \delta_0)>0$ and $\kappa=\kappa(d, \delta_0)$.

    \item to conclude, we notice that by \eqref{eqn:onePhaseHeightDerivativeErrorEst}
    \begin{equation}\label{eqn:frequencyMonotoneOnePhaseErrorse_H}
        \left\vert \frac{e_H}{H} \right\vert \le C \quad \text{for all } r\in(0, R)
    \end{equation}
    and some constant $C=C(d, \delta_0)>0$, $R=R(d, \delta_0)>0$. 
    
\end{itemize}
Combining \eqref{eqn:frequencyMonotoneOnePhaseEst1}, \eqref{eqn:frequencyMonotoneOnePhaseErrorE^{o,2}}, \eqref{eqn:frequencyMonotoneOnePhaseErrorsE^{o,1}G}, \eqref{eqn:frequencyMonotoneOnePhaseErrorse_I} and \eqref{eqn:frequencyMonotoneOnePhaseErrorse_H}, for all $r\in(0, R)$ and some constants $C=C(d, \delta_0)>0$, $R=R(d, \delta_0)>0$, $\kappa=\kappa(d, \delta_0)$ we have the estimate 
\begin{align*}
 \frac{d}{dr} \ln{N(r)}
 &\ge \frac{B(r) E^{o,2}(r)}{r F(r)H(r)}-\frac{2A(r)}{r^2F(r)D(r)}\left(G(r) + E^{o,1}(r)\right)+\frac{e_I(r)}{wrD(r)}-\frac{e_H(r)}{H(r)}\\
 &\ge -C D_i(r)^{\kappa-1} D_i'(r)-\frac{C\, D_i'(r)}{D_i(r)^2}\int_0^r H(\rho) \, d\rho -Cr^{\kappa - 1}=-g'(r), 
\end{align*}
where $g(r)$ is the function defined in \eqref{eqn:onePhaseFrequencyg(r)Def}.
In particular,
\[
\frac{d}{dr} e^{\,g(r)} N(r) \ge 0 \quad \text{for all } r\in(0, R)
\]
and a constant $R=R(d, \delta_0)>0$. We are only left to show that 
\[
\lim_{r\to0}g(r)=0.
\]
To this aim, we recall that the following properties hold:
\begin{itemize}
    \item[(i)] from \cref{assumptions:wC1aAprioriEst} and \cref{assumptions:startingAssumptions} we have
    \[
    \lim_{r\to0} D(r)=0,
    \]

    \item[(ii)] from \cref{lemma:weightedPoincare} and the monotonicity of $D_i(r)$ we have
    \[
    \frac{1}{D_i(r)} \int_0^r H(\rho) \, d\rho  \le C \frac{1}{D_i(r)} \int_0^r \rho D_i(\rho) \, d\rho \le Cr^2,
    \]

  \item[(iii)] using again \cref{lemma:weightedPoincare}, we have that 
  \begin{align*}
  \int_0^r \frac{H(\rho)}{D(\rho)} \, d\rho \le \int_0^r C\rho \, d\rho\le Cr^2, 
  \end{align*}
\end{itemize}
which imply \eqref{eqn:onePhaseFrequencyg(r)Limit0}. This concludes the proof.
\end{proof}

\subsection{Blow-up and proof of \cref{thm:onePhaseAnalyticObstacle}}
If $w \equiv 0$ in $B_1^+$, then thanks to \cref{remark:wNonDegeneracy} we have the degenerate case in \cref{thm:onePhaseAnalyticObstacle}, namely
\[
\partial \Omega_u \cap B_1 \equiv \mathrm{graph}(\phi) \cap B_1.
\]
Hence, in the following we can assume that $w \not\equiv 0$ in $B_1^+$, therefore we can work under \cref{assumptions:wC1aAprioriEst} and \cref{assumptions:startingAssumptions}.

Once the (almost-)monotonicity of the Almgren-type frequency function \eqref{eqn:onePhaseFrequencyN(r)Def} has been established in \cref{thm:freqmon}, we can proceed to prove the second case in \cref{thm:onePhaseAnalyticObstacle} via a blow-up analysis.

To begin with, we denote
\begin{align*}
\begin{split}
& \widetilde H(r) \coloneqq \lim_{\upsilon \to 1^-} H(r) = \int_{\partial B_r^+} \mu w^2, \\
& \widetilde D(r) \coloneqq \lim_{\upsilon \to 1^-} D(r) = \int_{B_r^+} M \nabla w \cdot \nabla w - \frac{1}{2} \int_{ B_r' } \partial_d Q(x) w^2, \\
& \widetilde N(r) \coloneqq \lim_{\upsilon \to 1^-} N(r) = \frac{r \widetilde D(r)}{\widetilde H(r)},
\end{split}
\end{align*}
where the height function $H(r)$ and the energy function $D(r)$ were defined in \eqref{eqn:onePhaseFrequencyHeightDef} and \eqref{eqn:onePhaseEnergyDef} respectively, $\upsilon$ being the parameter from \eqref{eqn:onePhaseCutoffDef}. Moreover, we introduce the limit function $\widetilde g(r)$  defined by
\[
\widetilde g(r) \coloneqq \lim_{\upsilon \to 1^-} g(r).
\]
By \cref{remark:freqMonotoneIndependent}, the frequency function $\widetilde N(r)$ satisfies
\begin{equation}\label{eqn:onePhaseFrequencyAlmostMonocotoneBlowUp}
e^{\widetilde g(r)} \widetilde N(r) \text{ is non-decreasing for all } r \in (0, R),
\end{equation}
where the constants $C, R, \kappa > 0$ are as in \cref{thm:freqmon}.

Let us define the Almgren rescalings of the solution $w$ as
\[
w_n(x) \coloneqq \frac{w(r_n x)}{\widetilde H(r_n)^{\sfrac12}} \quad \text{in } B_1,
\]
for some sequence of radii $r_n \to 0$.

Then, from \cref{corollary:boundaryEnergyControl} and a trace argument the sequence $w_n$ is bounded in $H^1(B_1^+)$ and, by \cref{prop:nonlinearThinObstacleC1aRegularity} also in $C^{1, \alpha}\left( B_{\sfrac12}^+ \right)$, uniformly in $n \in \N$. In particular, up to a subsequence
\[
w_n \to w_{\infty} \text{ weakly in } H^1(B_1) \text{ and in } C^{1, \alpha}\left( B_{\sfrac12}^+ \right)
\]
for a function $w_{\infty}: \R_+^d \to \R$. Taking sequences of the form $\rho r_n$, for all radii $\rho \in (0, 1)$, we see that the limit function satisfies:
\begin{itemize}
    \item $w_{\infty}$ is a nontrivial solution to the harmonic thin obstacle problem in $B_{\sfrac12}^+$. This follows from \cref{assumptions:startingAssumptions} and the $C^{1, \alpha}$ convergence.

    \item $w_{\infty}$ is positively $l-$homogeneous in $B_{\sfrac12}^+$, for some $l \in \R_{\ge 3/2}$. This follows in a standard way from \eqref{eqn:onePhaseFrequencyAlmostMonocotoneBlowUp}, since
    \[
    \frac{d}{dr}\left( \frac{r \int_{B_r^+} \vert \nabla w_{\infty} \vert^2}{\int_{\partial B_r^+} w_{\infty}^2} \right) = 0 \quad \text{ for all } r \in (0, +\infty).
    \]
\end{itemize}
Following the transformations from \cref{section:onePhaseFrequencyChangeOfCoordinates}, the set $\mathcal S_{1}(u)$ can be characterized, in terms of the function $w$, as
\[
\Sigma_w \coloneqq \{ w = 0 \} \cap \left\{ \frac{ M(x) \nabla w(x) \cdot e_d}{1+ \partial_d w(x)} - \frac{M(x) \nabla w(x) \cdot \nabla w(x)}{2 \left( 1 + \partial_d w (x) \right)^2} + \frac{1}{2} Q(x',0) = 0 \right\} \cap B_1'.
\]
Hence, let $\Sigma_n$ the singular sets for the blow-up sequence $w_n$
\[
\Sigma_n \coloneqq \{ w_n = 0 \} \cap \left\{ \frac{ M(r_n x) \nabla w_n(x) \cdot e_d}{1+ \partial_d w(r_n x)} - \frac{M(r_n x) \nabla w(r_n x) \cdot \nabla w_n(x)}{2 \left( 1 + \partial_d w (r_n x) \right)^2} + \frac{1}{2} Q(r_n x',0) = 0 \right\} \cap B_1'.
\]
and $\Sigma_{\infty}$ the singular set for the limit $w_{\infty}$
\[
\Sigma_{\infty} \coloneqq \{ w_{\infty} = 0 \} \cap \left\{\nabla w_{\infty} = 0 \right\} \cap B_1'.
\]
Now, without loss of generality, we can suppose by contradiction that the origin is a point of positive $\mathcal{H}_{\infty}^{d-2+\varepsilon}-$density for the singular set $\Sigma_w$ (see e.g. \cite[Proposition 11.3]{Giusti1984:MinimalSurfacesandBVFunctions}). From the $C^{1, \alpha}$ convergence of $w_n$ to $w_{\infty}$, we see that
\[
\mathcal{H}_{\infty}^{d-2+\varepsilon}\left( \Sigma_{\infty} \cap \overline{B_{1/4}} \right) \ge \lim \sup_{n \to +\infty} \mathcal{H}_{\infty}^{d-2+\varepsilon}\left( \Sigma_n \cap \overline{B_{1/4}} \right)
\]
Moreover, from the positive density assumption we have that
\[
\lim \sup_{n \to +\infty} \mathcal{H}_{\infty}^{d-2+\varepsilon} \left( \Sigma_n \cap \overline{B_{1/4}} \right) > 0,
\]
so that
\[
\mathcal{H}_{\infty}^{d-2+\varepsilon}\left( \Sigma_{\infty} \cap \overline{B_{1/4}} \right) > 0,
\]
which implies
\begin{equation}\label{eqn:onePhaseSingularSetContradiction}
\mathcal{H}^{d-2+\varepsilon}\left( \Sigma_{\infty} \cap B_{\sfrac12} \right) > 0.
\end{equation}
We can conclude the proof of \cref{thm:onePhaseAnalyticObstacle} combining \eqref{eqn:onePhaseSingularSetContradiction} with the following
\begin{lemma}\label{lem:dimensionreduction}
    Let the function $u:B_1^+ \to \R$ be a solution to the harmonic thin obstacle problem in $B_1^+ \subset \R^d$ and suppose that $u$ is positively $l-$homogeneous, for some $l \in \R_{\ge 3/2}$. Then, denoting with
    \[
       \Sigma_u \coloneqq \{ u = 0 \} \cap \left\{\nabla u = 0 \right\} \cap B_1',
    \]
    the following properties hold:
    \begin{enumerate}
        \item[(i)] if $d = 2$, then $\Sigma_u \cap B_1' = \{ 0 \}$, 

        \item[(ii)] if $d \ge 3$, then $\dim_{\mathcal H}(\Sigma_u )\leq d-2$.
    \end{enumerate}
\end{lemma}
\begin{proof}
We treat the two cases in order.
\begin{itemize}
    \item[(i)] A direct computation shows that the functions
    \begin{align*}
        & r^{k} cos(k \theta) \quad \text{for } k = 2n, \quad n \in \N_{\ge 1}, \\
        - & r^{k} sin(k \theta) \quad \text{for } k = 2n-1 \text{ and } k = 2n-\sfrac12, \quad n \in \N_{\ge 1},
    \end{align*}
    are the only admissible positively homogeneous solutions to the harmonic thin obstacle problem in $B_1^+$, which immediately implies case (i).

    \item[(ii)] We proceed by contradiction, via a standard dimension reduction argument. 
    
    Suppose that for some constant $\varepsilon > 0$
    \[
    \mathcal{H}^{d-2+\varepsilon}\left( \Sigma_u \right) > 0.
    \]
    In particular, we can find a point $x_0 \neq 0 \in B_1'$ with positive $\mathcal{H}_{\infty}^{d-2+\varepsilon}-$density. Since the function $u$ is a solution to the harmonic thin obstacle problem in $B_1^+$, the Almgren monotonicity formula
    \[
    N(r) \coloneqq \frac{r \int_{B_r(x_0)} \vert \nabla u \vert^2}{\int_{\partial B_r(x_0)} u^2}
    \]
    is monotone non-decreasing for all $r \in (0, r_0)$, where $r_0>0$ is a positive constant. In particular, up to a subsequence, the rescalings
    \[
    u_n \coloneqq \frac{u(x_0 + r_n x)}{\int_{\partial B_{r_n}(x_0)} u^2}
    \]
    converge weakly in $H^1( B_1^+)$ and strongly in $C^{1, \alpha}(B_r^+)$, for some $0< \alpha < 1$ and all $0< r < 1$, to a nontrivial function $u_{\infty}:B_1^+ \to \R$ with the following properties:
    \begin{itemize}
        \item[1.] $u_{\infty}$ is a solution to the harmonic thin obstacle problem in $B_1^+$;

        \item[2.] $u_{\infty}$ is positively $l-$homogeneous, for some $l \in \R_{\ge 3/2}$;

        \item[3.] denoting with $\Sigma_{u_{\infty}}$ the set
        \[
         \Sigma_{u_{\infty}} \coloneqq \{ u_{\infty} = 0 \} \cap \left\{\nabla u_{\infty} = 0 \right\} \cap B_1',
        \]
        we have that
        \begin{equation}\label{eqn:secondAlmgrenBlowUpPositiveMeasure}
        \mathcal{H}^{d-2+\varepsilon}\left( \Sigma_{u_{\infty}} \right) > 0;
        \end{equation}
        \item[4.] let $\nu \coloneqq x_0/\vert x_0 \vert$, then
        \begin{equation}\label{eqn:secondAlmgrenBlowUpInvariant}
        u_{\infty}(x) = u_{\infty}(x + t \nu) \quad \text{in } B_1^+.
        \end{equation}
    \end{itemize}
    The first three properties are clear. Concerning \eqref{eqn:secondAlmgrenBlowUpInvariant}, the following estimate holds for all points $x$ and $x + t \nu \in B_1 \cap \{ x_d \ge 0 \}$: 
    \begin{align}\label{eqn:secondAlmgrenBlowUpEst}
    \begin{split}
        \left\vert u_{\infty}(x) - u_{\infty}(x + t \nu) \right\vert &= \lim_{n \to +\infty} \frac{1}{\int_{\partial B_{r_n}(x_0)} u^2} \left\vert u(x_0 + r_n x) - u(x_0 + r_n x + r_n t \nu) \right\vert \le \\
        & \le \lim_{n \to +\infty} \frac{1}{\int_{\partial B_{r_n}(x_0)} u^2} \left\vert u(x_0 + r_n x) - u(x_0 + r_n \beta_n x + r_n t \nu) \right\vert + \\
        & + \lim_{n \to +\infty} \frac{1}{\int_{\partial B_{r_n}(x_0)} u^2} \left\vert u(x_0 + r_n \beta_n x + r_n t \nu) - u(x_0 + r_n x + r_n t \nu) \right\vert,
    \end{split}
    \end{align}
    where the quantity
    \[
    \beta_n \coloneqq \frac{\left\vert x_0 + r_n t \nu \right\vert}{\left\vert x_0 \right\vert}
    \]
    is such that the vectors
    \[
    x_0 + r_n x \quad \text{and} \quad x_0 + r_n t \nu + r_n \beta_n x
    \]
    are parallel to each other.

    Now, from the positive $l-$homogeneity of the function $u$ and the $C^{1, \alpha}$ convergence to $u_{\infty}$ we have that
    \begin{align}\label{eqn:secondAlmgrenBlowUpHomogeneityEst}
    \begin{split}
    & \frac{1}{\int_{\partial B_{r_n}(x_0)} u^2} \left\vert u(x_0 + r_n x) - u(x_0 + r_n \beta_n x + r_n t \nu) \right\vert = \\
    & = \frac{\left\vert \left\vert x_0 + r_n x\right\vert^{l} - \left\vert x_0 + r_n \beta_n x + r_n t \nu \right\vert^{l} \right\vert}{\left\vert x_0 + r_n x\right\vert^{l}} \frac{\left\vert u(x_0 + r_n x) \right\vert}{\int_{\partial B_{r_n}(x_0)} u^2} \to 0 \quad \text{as } n \to +\infty.
    \end{split}
    \end{align}
    Moreover, since $\beta_n \to 1$ as $n \to +\infty$, once again from the $C^{1, \alpha}$ convergence of $u$ to $u_{\infty}$ we have
    \begin{align}\label{eqn:secondAlmgrenBlowUpBeta_nEst}
    \begin{split}
    & \frac{1}{\int_{\partial B_{r_n}(x_0)} u^2} \left\vert u(x_0 + r_n \beta_n x + r_n t \nu) - u(x_0 + r_n x + r_n t \nu) \right\vert \to 0 \quad \text{as } n \to +\infty.
    \end{split}
    \end{align}
    Combining \eqref{eqn:secondAlmgrenBlowUpEst} with \eqref{eqn:secondAlmgrenBlowUpHomogeneityEst} and \eqref{eqn:secondAlmgrenBlowUpBeta_nEst} we get \eqref{eqn:secondAlmgrenBlowUpInvariant}.

    Now let us introduce the $d-1$ dimensional half ball 
    \[
    \Gamma \coloneqq \{ x \in \R^d : (x - x_0) \cdot \nu = 0 \} \cap B_1^+
    \]
    and the function $v:\Gamma \subset \R^{d-1} \to \R$ defined as the restriction of $u_{\infty}$ to $\Gamma$. In particular, thanks to \eqref{eqn:secondAlmgrenBlowUpInvariant}, the function $v$ satisfies:
    \begin{itemize}
        \item[1.] $v$ is a solution to the harmonic thin obstacle problem in $\Gamma$;

        \item[2.] $v$ is positively $l-$homogeneous, for some $l \in \R_{\ge 3/2}$;

        \item[3.] denoting with $\Sigma_{v} \subset \R^{d-1}$ the set
        \[
         \Sigma_{v} \coloneqq \{ v = 0 \} \cap \left\{\nabla v = 0 \right\} \cap \partial \Gamma \cap B_1',
        \]
        we have that
        \begin{equation*}
        \mathcal{H}^{d-3+\varepsilon}\left( \Sigma_{v} \right) > 0;
        \end{equation*}
    \end{itemize}
    Now, if we iterate the above argument $d-2$ times, we reach a contradiction with point (i). This concludes the proof of point (ii).\qedhere
\end{itemize}
\end{proof}

\appendix

\section{$C^{1, \alpha}$ regularity: proof of \cref{prop:nonlinearThinObstacleC1aRegularity}}\label{appendix:C1aRegularityNonlinearSignorini}
In this appendix we recall the proof to \cref{prop:nonlinearThinObstacleC1aRegularity}. These ideas are already present in the literature (e.g. \cite{RulandShi:C1aThinObstacleBlowUp, DiFazioSpadaro2022:NonlinearThinObstacleC1a}). However, since slight modifications are in order, for the sake of self-containment we prefer to give a little more details concerning the proofs.

We will use the notation
\begin{align}
L(x,w,p)&:=p\cdot M(x)p + (e_d\cdot p)\,w\,\partial_d Q(x)+\sum_{k=1}^3g_k(x,w,p)w^{3-k}\,P_k(p),
\end{align}
for every $x\in\R^d$, $w\in\R$, $p\in\R^d$, where $P_k$ and $g_k$ are the $k$-homogeneous polynomials and smooth functions from \cref{prop:changeofcoord}.\bigskip

From \cref{thm:thinobs} we deduce that a function $w$ which solves \eqref{eqn:thinObstacleRobinRegularity} is also a solution to the variational inequality
\begin{equation}\label{eqn:onePhaseNonlinearThinObstacleVariationalIneq}
    \int_{ B_1^+ } \nabla_p   L\left( x, w, \nabla w \right) \cdot \nabla  \eta + \partial_y L\left( x, w, \nabla w \right)  \eta \ge 0
\end{equation}
for all functions 
\begin{equation}\label{eqn:onePhaseNonlinearThinObstacleAdmissible}
\eta \in H^1(B_1) \text{ with } \quad \eta = 0 \text { on } \partial B_1 \cap \{ x_d \ge 0 \} \quad \text{ and } \quad w + \eta \ge 0 \text{ on } B_1'.
\end{equation}

We begin with the H\"{o}lder continuity for solutions of \eqref{eqn:onePhaseNonlinearThinObstacleVariationalIneq}.
\begin{lemma}[$C^{0, \alpha}$ regularity]\label{lemma:nonlinearThinObstacleC^{0, a}Reg}
     Let the function $w:B_1 \to \R$ be a solution of \eqref{eqn:thinObstacleRobinRegularity} under \cref{assumptions:wC1aAprioriEst}. Then, there exist constants $C = C(d, \delta_0)>0$ and $\alpha = \alpha(d, \delta_0) \in (0, 1)$ such that
     \[
     \| w \|_{C^{0, \alpha}\left( \overline{B_{\sfrac12}^+} \right)} \le C \| w \|_{L^2\left( B_1^+ \right)}.
     \]
\end{lemma}
\begin{proof}
We only sketch the proof, following the work \cite{BeiraoDaVeigaConti1972:ThinObstacleC0a}. This lemma is essentially a consequence of the De Giorgi-Nash-Moser iterations. In particular, let
\[
A(k, r, x) \coloneqq \{ w \ge k \} \cap B_r(x) \cap B_1^+, \quad B(k, r, x) \coloneqq \{ w \le k \} \cap B_r(x) \cap B_1^+ \quad \text{ for all } k \in \R.
\]
Once one can prove the Caccioppoli-Leray inequalities on the super/sublevel sets $A(k, r)$ and $B(k, r)$
\begin{align}\label{eqn:nonlinearThinObstacleCaccioppoliLerayIneq}
\begin{split}
& \int_{B(k, r)} \vert \nabla w \vert^2 \le \frac{C}{(R-\rho)^2} \int_{B(k, R)} (w-k)^2 \qquad \text{for all } k \in \R \\
& \int_{A(k, r)} \vert \nabla w \vert^2 \le \frac{C}{(R-\rho)^2} \int_{A(k, R)} (w-k)^2 \qquad \text{for all } k \ge 0
\end{split}
\end{align}
for all radii $0 < r < R$, all points $x \in B_{\sfrac12}$ such that $B_R(x) \Subset B_1$ and a constant $C=C(d, \delta) > 0$, the H\"{o}lder regularity follows for instance from the iterations in \cite[Section 4]{BeiraoDaVeigaConti1972:ThinObstacleC0a}. However, the inequalities \eqref{eqn:nonlinearThinObstacleCaccioppoliLerayIneq} follow in the usual way, testing the variational inequality \eqref{eqn:onePhaseNonlinearThinObstacleVariationalIneq} with (admissible) competitors of the form
\[
\eta = - \varphi (w-k)^- \quad \text{and} \quad \eta = - \varphi (w-k)^+.
\]
where $\varphi$ is a radial non-negative cutoff function with support in $B_{r}(x)$.
\end{proof}
Before stating the $C^{1, \alpha}$ regularity, let us consider the functions
\[
w_r \coloneqq \frac{w(r x)}{r} \quad \text{and} \quad \zeta \coloneqq \frac{w_r}{b}, \text{ with } b \coloneqq \| w_r \|_{L^2(B_1)}.
\]
Then, in terms of the function $\zeta$ we can rewrite explicitly \eqref{eqn:onePhaseNonlinearThinObstacleVariationalIneq} as
\begin{align}\label{eqn:onePhaseNonlinearThinObstacleVariationalIneqLinearized}
\begin{split}
    & \int_{B_1^+} 2 M(r x', r x_d + w(r x) ) \frac{\nabla \zeta}{1+ \partial_d w(rx)} \cdot \nabla \eta\\
    &\qquad- \int_{B_1^+} \left[ \frac{1}{\left( 1 + \partial_d w (rx) \right)^2} M(r x', r x_d+w(rx)) \nabla w(rx) \cdot \nabla \zeta \right] \partial_d \eta \\
    &\qquad\qquad + \int_{B_1^+} \frac{r}{1 + \partial_d w(r x)} \left[\partial_d M \right] \left( r x', r x_d + w(rx) \right) \nabla w (rx) \cdot \nabla \zeta \, \eta\\
    &\qquad\qquad\qquad+ \int_{B_1'} \frac{1}{b} Q(r x', r x_d + r b \zeta) \, \eta \ge 0,
\end{split}
\end{align}
for all functions $\eta \in H^1(B_1)$ with
\begin{equation*}
\eta = 0 \text { on } \partial B_1^+ \quad \text{ and } \quad \zeta + \eta \ge 0 \text{ on } B_1'.
\end{equation*}
Now, the $C^{1, \alpha}$ regularity of solutions can be adapted from the linearization method developed in \cite{RulandShi:C1aThinObstacleBlowUp}. The application of this method will be almost straightforward as, apart from the $C^{0, \alpha}$ regularity in \cref{lemma:nonlinearThinObstacleC^{0, a}Reg}, it essentially only requires to write down the rescaled variational inequality \eqref{eqn:onePhaseNonlinearThinObstacleVariationalIneqLinearized} which, at the free boundary, linearizes to a constant coefficients thin obstacle problem. More precisely, we have the following decay property at the branching points.
\begin{lemma}[$C^{1, \alpha}$ decay at branching points]\label{lemma:nonlinearThinObstacleC^{1, a}Reg}
    Let w be a solution to \eqref{eqn:onePhaseNonlinearThinObstacleVariationalIneq} under \cref{assumptions:wC1aAprioriEst} and assume that the origin is a branching point. Then, there exist constants $\beta=\beta(d, \delta_0) \in (0, \sfrac12)$ and $C = C(d, \delta_0)>0$ such that
    \begin{equation}\label{eqn:nonlinearThinObstacleLocalC1aFreeBd}
    \| w \|_{L^2(B_r^+)} \le C r^{1+\beta} \| w \|_{L^2\left( B_1^+ \right)} \quad \text{ for all } r \in (0, \sfrac12).
    \end{equation}
\end{lemma}
\begin{proof}
    The proof follows by the linearization method in \cite{RulandShi:C1aThinObstacleBlowUp}. 
    We sketch the main steps of the proof and refer to \cite[Section 2]{RulandShi:C1aThinObstacleBlowUp} for more details. 

    Proceeding by contradiction, one can find a sequence of solutions $u_k$ to \eqref{eqn:onePhaseNonlinearThinObstacleVariationalIneq} with $u_k(0) = 0$ and radii $r_k \to  0$ such that
    \begin{equation}\label{eqn:onePhaseNonlinearThinObstacleC1aContrad}
         \| w_k \|_{L^2(B_{r_k}^+)} \ge k {r_k}^{1+\beta} \| w_k \|_{L^2\left( B_1^+ \right)}.
    \end{equation}
    Moreover, without loss of generality one can assume that
    \begin{equation}\label{eqn:nonlinearThinObstacleLocalC1ar_kChoice}
    r_k \coloneqq \sup \left\{ r \in (0, \sfrac12) : \text{ condition } \eqref{eqn:onePhaseNonlinearThinObstacleC1aContrad} \text{ holds true} \right\}.    
    \end{equation}
    Since \cref{lemma:nonlinearThinObstacleC^{0, a}Reg} implies 
    \begin{equation}\label{eqn:nonlinearThinObstacleLocalC1ar_kto0}
    r_k \to 0^+,
    \end{equation}
    from \eqref{eqn:nonlinearThinObstacleLocalC1ar_kChoice} and \eqref{eqn:nonlinearThinObstacleLocalC1ar_kto0} the $L^2$-normalized functions
    \[
    w_k(x) \coloneqq \frac{u_k(r_k x)}{r_k} \frac{r_k}{\| 
     u_k(r_k x)\|_{L^2(B_1^+)}}
    \]
    converge, thanks to \cref{lemma:nonlinearThinObstacleC^{0, a}Reg} and \eqref{eqn:onePhaseNonlinearThinObstacleVariationalIneqLinearized}, weakly in $H^1(B_r^+)$ and in $C^{0, \alpha}(B_r^+)$ for all $r > 0$, to a (nontrivial) global solution $w_{\infty}$ of the harmonic thin obstacle problem, with $w(0) = 0$. Moreover, from \eqref{eqn:nonlinearThinObstacleLocalC1ar_kChoice} and \cref{lemma:nonlinearThinObstacleC^{0, a}Reg} there exists a constant $C=C(d, \delta_0)>0$ such that
    \begin{equation}\label{eqn:nonlinearThinObstacleBlowUp1Growth}
    \|w_{\infty}\|_{L^\infty(\partial B_r^+)} \le C r^{1 + \beta}\quad\text{for every}\quad r\ge 1.
    \end{equation}
    Combining \eqref{eqn:nonlinearThinObstacleBlowUp1Growth} with the Liouville-type lemma \cite[Lemma 3]{RulandShi:C1aThinObstacleBlowUp} implies that (without loss of generality we can assume even symmetry for $ w_{\infty}$)
    \[
    w_{\infty} = \gamma x_d \text{ for some } \gamma < 0.
    \]
    In particular, the following property holds:
    \[
    \lim_{k \to +\infty} \inf_{\gamma \le 0} \frac{\| w_k - \gamma x_d \|_{L^2\left( B_r^+ \right)}}{\| w_k\|_{L^2\left( B_r^+ \right)}} = 0,
    \]
    for all $r > 0$.

    Now let us consider the sequence
    \[
    \widetilde w_k (x; s_k) \coloneqq \frac{u_k(r_k s_k x + x_k)}{r_k s_k} \frac{r_k s_k}{\| 
     u_k(r_k s_k x+x_k)\|_{L^2(B_1^+)}}
    \]
    for a sequence of points $x_k \in B_1'$ and of real numbers $s_k \in (0, 1)$. Choosing points $x_k \in B_1'$ such that
    \begin{itemize}
        \item[(i)] $w_k(x_k) > 0$,

        \item[(ii)] $\vert x_k \vert^{\alpha} \le \frac{1}{k} r_k^{1+\beta}$,
    \end{itemize}
    which is possible since by assumption the origin is a branching point for the free boundary, one can guarantee that
    \begin{equation}\label{eqn:nonlinearThinObstacleChangeCenter}
    \lim_{k \to +\infty} \inf_{\gamma \le 0} \frac{\| \widetilde w_k (x; 1) - \gamma x_d \|_{L^2\left( B_1^+ \right)}}{\| \widetilde w_k (x; 1)\|_{L^2\left( B_1^+ \right)}} = 0,
    \end{equation}
    thanks to \eqref{eqn:nonlinearThinObstacleLocalC1ar_kChoice}. On the other hand, for each fixed $k \in \N$ and sequence $\delta_l \to 0^+$, 
    \begin{equation}\label{eqn:nonlinearThinObstacleInterior0}
    \lim_{l \to +\infty} \inf_{\gamma \le 0} \frac{\| \widetilde w_k (x; \delta_l ) - \gamma x_d \|_{L^2\left( B_1^+ \right)}}{\| \widetilde w_k (x; \delta_l)\|_{L^2\left( B_1^+ \right)}} = 1
    \end{equation}
    since, thanks to point (i) and elliptic regularity for Robin problems, one has that the functions $\widetilde w_k (x; \delta_l )$ converge in $C^{1, \alpha}(B_r^+)$ for all $r > 0$ to a nontrivial constant function $\widetilde w_{\infty}$. In particular, from \eqref{eqn:nonlinearThinObstacleChangeCenter} and \eqref{eqn:nonlinearThinObstacleInterior0} one has that for each real number $\mu \in (0, 1)$ there exists a sequence $s_k \in (0, 1)$ such that
    \begin{equation}\label{eqn:nonlinearThinObstacleInterior}
    \lim_{k \to +\infty} \inf_{\gamma \le 0} \frac{\| \widetilde w_k (x; s_k ) - \gamma x_d \|_{L^2\left( B_1^+ \right)}}{\| \widetilde w_k (x; s_k)\|_{L^2\left( B_1^+ \right)}} = \mu.
    \end{equation}
    Let us choose the sequence $s_k$ as
    \begin{equation}\label{eqn:nonlinearThinObstacleLocalC1as_kChoice}
    s_k \coloneqq \sup \left\{ s \in (0, 1) : \inf_{\gamma \le 0} \frac{\| \widetilde w_k (x; s_k ) - \gamma x_d \|_{L^2\left( B_1^+ \right)}}{\| \widetilde w_k (x; s_k)\|_{L^2\left( B_1^+ \right)}} \le \mu \text{ for all } s \ge s_k \right\}.    
    \end{equation}
    so that \eqref{eqn:nonlinearThinObstacleInterior} still holds true. Now one can show that
    \begin{equation}\label{eqn:nonlinearThinObstacleLocalC1as_kTo0}
    s_k \to 0^+ \text{ as } k \to +\infty.
    \end{equation}
    Indeed, suppose by contradiction that \eqref{eqn:nonlinearThinObstacleLocalC1as_kTo0} fails, so that $s_k \to s_0 > 0$. Since 
    \[
    \widetilde w_k (x; s_k ) = w_k \left( s_k x + \frac{x_k}{r_k} \right) \frac{ \| 
     u_k(r_k x)\|_{L^2(B_1^+)} }{\| 
     u_k(r_k s_k x+x_k)\|_{L^2(B_1^+)}},
    \]
    thanks to point (ii) one has that
    \[
     \widetilde w_k (x; s_k ) \to w_{\infty}(s_0 x) = \gamma s_0 x_d,
    \]
    weakly in $H^1(B_r^+)$ and in $C^{0, \alpha}(B_r^+)$ for all $r > 0$ and
    \[
    \frac{ \| 
    u_k(r_k x)\|_{L^2(B_1^+)} }{\| 
    u_k(r_k s_k x+x_k)\|_{L^2(B_1^+)}} \to \frac{1}{\| 
    w_{\infty}(s_0 x)\|_{L^2(B_1^+)}} > 0.
    \]
    This, however, is in contradiction with \eqref{eqn:nonlinearThinObstacleInterior}, so that \eqref{eqn:nonlinearThinObstacleLocalC1as_kTo0} holds true.

    Now, from the choice \eqref{eqn:nonlinearThinObstacleLocalC1as_kChoice} and the proof of \cite[Lemma 2.7]{RulandShi:C1aThinObstacleBlowUp}, choosing $\mu>0$ sufficiently small, one has that
    \begin{equation}\label{eqn:nonlinearThinObstacleLocalC1as_kGlobal}
    \| \widetilde w_k(x; s_k)\|_{L^2(B_R^+)} \le C R^{1+\beta} \text{ for all } R \in (2, \sfrac{s_k}{2}).
    \end{equation}
    for some constant $C = C(d, \delta)>0$. In particular, combing \eqref{eqn:nonlinearThinObstacleLocalC1as_kGlobal} with \cref{lemma:nonlinearThinObstacleC^{0, a}Reg}, \eqref{eqn:nonlinearThinObstacleLocalC1as_kTo0} and \eqref{eqn:onePhaseNonlinearThinObstacleVariationalIneqLinearized} one has that
    \[
    \widetilde w_k(x; s_k) \to \widetilde w_{\infty}
    \]
    weakly in $H^1(B_r^+)$ and in $C^{0, \alpha}(B_r^+)$ for all $r > 0$, to a nontrivial global solution $\widetilde w_{\infty}$ of the harmonic thin obstacle problem, with the estimate
    \[
    \| \widetilde w_{\infty}\|_{L^{\infty}(\partial B_R^+)} \le C R^{1+\beta}  \text{ for all } R \ge 2.
    \]
    Hence, using once again \cite[Lemma 3]{RulandShi:C1aThinObstacleBlowUp}, assuming without loss of generality even symmetry for $ w_{\infty}$, one can deduce that
    \[
    \widetilde w_{\infty} = c + \gamma' x_d \text{ for some constants } c \in \R \text{ and } \gamma' \le 0.
    \]
    The aim now is to show that,
    \begin{equation}\label{eqn:nonlinearThinObstacleLocalC1Contrad}
    c = 0 \text{ and } \gamma' < 0.
    \end{equation}
    First, one can show by contradiction that
    \begin{equation}\label{eqn:nonlinearThinObstacleLocalC1DistBRanching}
    r_k s_k > \dist\left( x_k, \partial \{ u_k = 0 \} \cap B_1' \right).
    \end{equation}
    Indeed, if \eqref{eqn:nonlinearThinObstacleLocalC1DistBRanching} were false, by elliptic regularity for Robin problems $\gamma' = 0$ and consequently $c > 0$. This, however, would contradict that $\mu < 1$ in \eqref{eqn:nonlinearThinObstacleInterior}, so that \eqref{eqn:nonlinearThinObstacleLocalC1DistBRanching} holds true.  Consequencly, there exists $x' \in B_1'$ such that $\widetilde w_{\infty}(x') = 0$, which implies $c = 0$. Since $\widetilde w_{\infty}$ is nontrivial, one also has $\gamma' < 0$. This proves \eqref{eqn:nonlinearThinObstacleLocalC1Contrad}.
    
    Finally, conditions \eqref{eqn:nonlinearThinObstacleLocalC1Contrad} are in contradiction with the fact that $\mu > 0$ in \eqref{eqn:nonlinearThinObstacleInterior}, so that the proof is concluded.
\end{proof}
We are ready for the
\begin{proof}[Proof of \cref{prop:nonlinearThinObstacleC1aRegularity}.]
    Combining \cref{lemma:nonlinearThinObstacleC^{0, a}Reg}, \cref{lemma:nonlinearThinObstacleC^{1, a}Reg} together with interior elliptic regularity and elliptic regularity for Robin problems, \cref{prop:nonlinearThinObstacleC1aRegularity} follows via a standard projection argument. We refer for instance to \cite{RulandShi:C1aThinObstacleBlowUp} and the references therein.
\end{proof}

\section*{Acknowledgements} LS acknowledges the support of the NSF Career Grant DMS 2044954. LF and BV are supported by the European Research Council (ERC), under the European Union's Horizon 2020 research and innovation program, through the project ERC VAREG - {\em Variational approach to the regularity of the free boundaries} (No.\,853404). LF and BV acknowledge the MIUR Excellence Department Project awarded to the Department of Mathematics, University of Pisa, CUP I57G22000700001. LF is a member of INdAM-GNAMPA. BV acknowledges support from the projects PRA 2022 14 GeoDom (PRA 2022 - Università di Pisa) and MUR-PRIN ``NO3'' (No. 2022R537CS). The authors would like to thank Camillo De Lellis and Guido De Philippis for the stimulating discussions on the boundary Almgren's frequency function for area-minimizing currents.


\bibliographystyle{plain}
\bibliography{FreeBoundary_bib.bib}


\medskip
\small
\begin{flushright}
\noindent 
\verb"lorenzo.ferreri@sns.it"\\
Classe di Scienze, Scuola Normale Superiore\\ 
piazza dei Cavalieri 7, 56126 Pisa (Italy)
\end{flushright}

\begin{flushright}
\noindent 
\verb"lspolaor@ucsd.edu"\\
Department of Mathematics, UC San Diego,\\
 AP\&M, La Jolla, California, 92093, USA
\end{flushright}

\begin{flushright}
\noindent 
\verb"bozhidar.velichkov@unipi.it"\\
Dipartimento di Matematica, Università di Pisa\\ 
largo Bruno Pontecorvo 5, 56127 Pisa (Italy)
\end{flushright}

\end{document}